\documentclass[11pt,oneside,reqno]{amsart}
\usepackage{amsmath, amssymb, amsthm}
\usepackage{url}
\usepackage{mathrsfs}
\usepackage[ansinew]{inputenc}
\usepackage[breaklinks]{hyperref}
\usepackage{multirow}

\allowdisplaybreaks

\newcommand{\g}{\gamma}

\renewcommand{\(}{\left\(}
\renewcommand{\)}{\right\)}
\renewcommand{\[}{\left\[}
\renewcommand{\]}{\right\]}
\newtheorem{remark}[]{Remark}
\numberwithin{equation}{section}
 \theoremstyle{plain}
\newtheorem{theorem}{Theorem}[section]
\newtheorem{lemma}[theorem]{Lemma}
\newcommand{\sm}{\left(\begin{smallmatrix}}
\newcommand{\esm}{\end{smallmatrix}\right)}

\newtheorem{corollary}[theorem]{Corollary}
\newtheorem{proposition}[theorem]{Proposition}

   \makeatletter
\def\proof{\@ifnextchar[{\@oproof}{\@nproof}}
\def\@oproof[#1][#2]{\trivlist\item[\hskip\labelsep\textit{#2 Proof of\
#1.}~]\ignorespaces}
\def\@nproof{\trivlist\item[\hskip\labelsep\textit{Proof.}~]\ignorespaces}

\makeatother

\usepackage{color}
\usepackage{amsmath}

\definecolor{blue}{rgb}{0,0,1}
\definecolor{red}{rgb}{1,0,0}
\definecolor{green}{rgb}{0,.6,.2}
\definecolor{purple}{rgb}{1,0,1}

\long\def\red#1\endred{{\color{red}#1}}
\long\def\blue#1\endblue{{\color{blue}#1}}
\long\def\purple#1\endpurple{{\color{purple}#1}}
\long\def\green#1\endgreen{{\color{green}#1}}

\begin{document}
\title{Arithmetic properties of the Herglotz-Zagier-Novikov function}
\author{YoungJu Choie}
\address{Department of Mathematics, Pohang Universtity of Science and Technology, \\
Pohang, Republic of Korea}
\email{yjc@postech.ac.kr}
\author{Rahul Kumar}
\address{Department of Mathematics, The Pennsylvania State University, University Park, PA 16802, U.S.A.}
\email{rjk6031@psu.edu} 


\subjclass[2000]{Primary 11F67, 11F99 Secondary 30D05, 33E20}
  \keywords{functional equations, Herglotz function, dilogarithm, cohomology, evaluations of integrals}
\maketitle
\pagenumbering{arabic}
\pagestyle{headings}
\begin{abstract}
In this article, we undertake the study of the function $\mathscr{F}(x;u,v)$, which we refer to as the \emph{Herglotz-Zagier-Novikov function}. This function appears in Novikov's work on the Kronecker limit formula, which was motivated by Zagier's paper where he obtained the Kronecker limit formula in terms of the Herglotz function $F(x)$. Two, three, and six-term functional equations satisfied by  $\mathscr{F}(x;u,v)$ are exhibited. These are cohomological relations  coming from the action of an involution and  SL$_2(\mathbb{Z})$ on $\mathbb{C}\times {\mathbb{D}_1}^2$ (the unit circle ${\mathbb{D}_1})$.  We also provide the special values of $\mathscr{F}(x;u,v)$ at rational arguments of $x$. Importantly, $\mathscr{F}(x;u,v)$ serves as a unified generalization of three other interesting functions, namely $F(x)$, $J(x)$, and $T(x)$, which also appear  in various Kronecker limit formulas  and   are previously studied by Cohen, Herglotz, Muzaffar and Williams, and Radchenko and Zagier. Consequently, our study not only reveals the numerous elegant properties of $\mathscr{F}(x;u,v)$ but also helps us to further develop the theories of functions related to its special cases such as $J(x)$ and $T(x)$.     
\end{abstract}
\tableofcontents

\section{Introduction and Motivation}

Zagier's seminal paper \cite{zagier1975} on the Kronecker limit formula for real quadratic fields has been a source of inspiration for many researchers in the field of number theory. Among these, Novikov \cite{novikov}  obtained a Kronecker limit formula involving an interesting function that, until now, remained unexplored. In this paper, we initiate the study of properties of this function, which we now refer to as the \emph{Herglotz-Zagier-Novikov function}, and denote it by $\mathscr{F}(x;u,v)$. We define it as follows:
\begin{align}\label{novikov function}
\mathscr{F}(x;u,v):=\int_0^1\frac{\log\left(1-ut^x\right)}{v^{-1}-t}dt, \quad  \mathrm{Re}(x)>0,
\end{align}
where  $u\in\mathbb{C}\backslash(1,\infty), u\neq0$ and $v\in\mathbb{C}\backslash[1,\infty),  v\neq0$. The main objective of this paper is to showcase the  multitude of  intriguing  properties inherent in the function $\mathscr{F}(x;u,v)$. Moreover,  using its properties, we further develop the theories of other functions that have been studied previously in the literature.

The Herglotz-Zagier-Novikov function $\mathscr{F}(x;u,v)$ serves as a unified generalization of three functions studied by Herglotz \cite{her1}, Zagier \cite{zagier1975}, and Muzaffar and Williams \cite{muzwil} in their work on various Kronecker limit formulas. To begin, we  define  these individual functions:
\begin{align}\label{herglotz defn}
F(x ):=\int_0^1\left(\frac{1}{1-t}+\frac{1}{\log(t)}\right)\log(1-t^x)\frac{dt}{t},  
\end{align}
\begin{align}\label{j(x)}
J(x):=\int_0^1\frac{\log(1+t^x)}{1+t}dt,
\end{align}
and 
\begin{align}\label{t(x) defn}
T(x):=\int_0^\infty\frac{\tan^{-1}(t^x)}{1+t^2}dt,  
\end{align}
for $\mathrm{Re}(x)>0$. The function $F(x)$ appears in the work of Herglotz \cite{her1} in connection with the Kronecker limit formula for real quadratic fields. 
It can be analytically continued through the expression \cite[Equation (1.2)]{raza}
\begin{align}\label{herglotz defn ser}
F(x)=\sum_{n\geq 1}\frac{\psi(nx)-\log(nx)}{n} \quad (x\in \mathbb{C}':=\mathbb{C}\backslash (-\infty,0]),
\end{align}
where $\psi(x):=\Gamma'(x)/\Gamma(x)$ denotes  the digamma function.
Later on, in his breakthrough work on the Kronecker limit formula for real quadratic fields, Zagier \cite{zagier1975} rediscovered the function $F(x)$  and also initiated its systematic study. In particular, the following  two-term and three-term   cohomological relations, namely, the  period relations on SL$_2(\mathbb{Z})$,   were given by Zagier \cite[Section 7]{zagier1975}:
\begin{align}
&F(x)+F\left(\frac{1}{x}\right)=-2\left(\frac{1}{2}\g^2+\frac{\pi^2}{12}+\g_1\right)+\frac{1}{2}\log^{2}(x)-\frac{\pi^2}{6x}(x-1)^2,\label{fe2}\\
&F(x)-F(x+1)-F\left(\frac{x}{x+1}\right)=\frac{1}{2}\g^2+\frac{\pi^2}{12}+\g_1+\textup{Li}_2\left(\frac{1}{1+x}\right),\label{fe1}
\end{align}
where $\g$ is the Euler's constant and $\g_1$ is the first Stieltjes constant. Here, and throughout the paper, $\mathrm{Li}_2(u)$ denotes the dilogarithm function, which is a special case of the polylogarithm function \cite{lewin}, \cite[p.~5]{zagierdilog}
\begin{align*}
\mathrm{Li}_s(z):=\sum_{n=1}^\infty\frac{z^n}{n^s} \qquad (z,s\in\mathbb{C},\ |z|<1).
\end{align*}
Very recently, Radchenko and Zagier \cite{raza} studied the functions $F(x)$ and $J(x)$ extensively and revealed their connections to Stark's conjecture, Hecke operators, and the cohomology of the modular group PSL$_2(\mathbb{Z})$. They also  found the following connection between the functions $F(x)$ and $J(x)$ \cite{raza}:
\begin{align}\label{jxfx}
J(x)=F(2x)-2F(x)+F\left(\frac{x}{2}\right)+\frac{\pi^2}{12x}.
\end{align} 
In their paper, the function $F(x)$ is referred to as the \emph{Herglotz function}, while Masri \cite{masri} named it the \emph{Herglotz-Zagier function}. Building upon these analogies, we call the function $\mathscr{F}(x;u,v)$  the \emph{Herglotz-Zagier-Novikov function}, as it appears in the work of Novikov and also it reduces to $F(x)$ as the limits of 
$u$ and $v$ approach $1$. More precisely:
\begin{align*}
\lim_{\substack{u,v\to1}}\left\{\mathscr{F}(x;u,v)-\frac{1}{x}\mathrm{Li}_2(u)+\log(1-u)\left(\log(1-v)+\gamma+\log(x)\right)+\mathrm{Li}_s'(u)|_{s=1}\right\}=F(x).
\end{align*}
We state this result precisely in Proposition \ref{connection $F$} below. 

The functions $J(x), F(x)$ and $T(x)$ also appeared in the Kronecker limit formula of a restricted Epstein zeta function studied by Muzaffar and Williams \cite[p.~41, Theorem 2]{muzwil}\footnote{$F(x)$ is present in the function $V(d)$ as it is easy to see, by using our \eqref{herglotz defn}, that the second integral in $V(d)$ is nothing but $F(\epsilon)-F(\epsilon')-\frac{\pi^2(\epsilon-\epsilon')}{6}$.}, where they also offered the special values of $J(x)$ and $T(x)$  at  some units of real quadratic fields \cite[Theorems 6--10]{muzwil}. 
The authors of \cite{raza} found explicit evaluations of $J(x)$ at several quadratic units of the real quadratic field \cite[pp.~244--245, Tables 1 and 2]{raza} with the help of the theory of the Herglotz function $F(x)$. 
Also, see books \cite[p.~145, Exercise 102(a) \& p.~271, Exercise (60)]{cohen} and \cite[pp.~132--133]{br}. A two-parameter generalization of $J(x)$ has been given recently in \cite[Theorem 2.8]{hhf}. 

Not only $F(x)$ is a special case of $\mathscr{F}(x;u,v)$, but the functions $J(x)$ and $T(x)$ can also be obtained from the function $\mathscr{F}(x;u,v)$:
\begin{align*}
\mathscr{F}(x;-1,-1)=-J(x),
\end{align*}
and
\begin{align*}
 \frac{1}{4}\left\{\mathscr{F}\left(x;i,i\right)+\mathscr{F}\left(x;-i,-i\right)-\mathscr{F}\left(x;i,-i\right)-\mathscr{F}\left(x;-i,i\right)\right\}=T(x).
\end{align*}
Refer to Sections \ref{jx section} and \ref{tx section} where we prove these relations and obtain further properties of $J(x)$ and $T(x)$.

After exploring the functions $F(x)$, $J(x)$, and $T(x)$, along with their connections to the function $\mathscr{F}(x;u,v)$, we proceed to examine the main properties of $\mathscr{F}(x;u,v)$ in subsequent sections. However, we outline our plan briefly here and elaborate on it in the upcoming sections.
By considering the relation  of element  $U^6=\sm 1 & -1\\ 1& 0 \esm^6=I $ of SL$_2(\mathbb{Z})$, we obtain   six-term cohomological period relations and involution $\sm 0 & 1\\ 1 & 0\esm$  gives two-term period relations. For instance, one of the functional equations that we obtain for $\mathscr{F}(x;u,v)$ is:
\begin{align}\label{2termint}
\mathscr{F}\left(x;u,v\right)+\mathscr{F}\left(\frac{1}{x};v,u\right)=-\log\left(1-u\right)\log\left(1-v\right).
\end{align}
See  Section \ref{fxuv} below for more details on the theory of $\mathscr{F}(x;u,v)$. Moreover, we provide special values of $\mathscr{F}(x;u,v)$ when $x$ is a rational number.

As each of the functions $J(x)$, $T(x)$, and $F(x)$ can be derived from $\mathscr{F}(x;u,v)$, as demonstrated above, we derive functional equations for these functions as well. For example, \eqref{2termint} gives the following new functional equations for $J(x)$ and $T(x)$ as special cases:
\begin{align*}
J(x)+J\left(\frac{1}{x}\right)&=\log^2(2),\\
T(x)+T\left(\frac{1}{x}\right)&=\frac{\pi^2}{16}.
\end{align*}
Refer to Sections \ref{jx section} and \ref{tx section} for more details on these relations. We here mention that while the functions $J(x)$ and $T(x)$ have been previously studied,  these simple functional equations surprisingly do not appear in the literature. Additionally, we offer explicit evaluations of $J(x)$ and $T(x)$ at natural numbers as special cases of $\mathscr{F}(x;u,v)$ in Sections \ref{jx section} and \ref{tx section}. 

After giving a glimpse of some of the properties of the Herglotz-Zagier-Novikov function $\mathscr{F}(x;u,v)$, we proceed to present our main results in the subsequent sections. 


\medskip

\section{Properties of the Herglotz-Zagier-Novikov function $\mathscr{F}(x;u,v)$}\label{fxuv}

Note that initially, we have defined the function $\mathscr{F}(x;u,v)$ for $\text{Re}(x)>0$ in \eqref{novikov function}. However, the following proposition provides its analytic continuation to a larger region.
\begin{proposition}\label{analytic continuation 1}
For non-zero complex numbers such that $|u|<1$ and $|v|<1$, $\mathscr{F}(x;u,v)$ can be extended analytically by
\begin{align}\label{analytic continuation}
\mathscr{F}(x;u,v)=-\sum_{m,n\geq1}^\infty\frac{u^mv^{n}}{m(mx+n)}, \quad x\in\mathbb{C}'=\mathbb{C}\backslash(-\infty,0].
\end{align}
\end{proposition}

To give further properties, let us denote
\begin{align}\label{set}
\mathbb{D}:=\left\{z\in\mathbb{C}\backslash\{0\}:|z|\leq1\right\},
\qquad \mathbb{D}':=\mathbb{D}\backslash\{1\},
\end{align}
\begin{align}\label{set1}
\mathbb{D}_1:=\left\{z\in\mathbb{C}  :  |z|=1 \right\} \qquad \mathrm{ and}\qquad {\mathbb{D}_1}':=\mathbb{D}_1\backslash\{1\}.
\end{align}

The function $\mathscr{F}\left(x;u, v\right)$ satisfies the following relations.
\begin{theorem}\label{sum formula for Jx}\textup{(}{\bf   Duplication formula}\textup{)}
Take $u\in\mathbb{D}$ and $v,v^2\in\mathbb{D}'.$  Then, for $\mathrm{Re}(x)>0$, the relations holds
\begin{align}
\mathrm{(1)}\qquad \mathscr{F}\left(2x;u^2,v\right)=\mathscr{F}\left(x;u,v\right)+\mathscr{F}\left(x;-u,v\right).\label{sum formula for Jx eqn1}\\
\mathrm{(2)}\qquad \mathscr{F}\left(\frac{x}{2};u,v^2\right)=\mathscr{F}\left(x;u,v\right)+\mathscr{F}\left(x;u,-v\right).\label{sum formula for Jx eqn2}
\end{align}
\end{theorem}



\subsection{Two, three and six-term functional equations  of $\mathscr{F}(x;u,v)$}
The Herglotz-Zagier-Novikov function $\mathscr{F}(x;u,v)$ satisfies the following functional equations similar to that of $F(x)$ given above in \eqref{fe2} and \eqref{fe1}.
\begin{theorem}\label{functional equations}
Let $\mathscr{F}(x;u,v)$ be defined in \eqref{novikov function} or \eqref{analytic continuation} and $x\in\mathbb{C}'$. 
\begin{enumerate}
\item For $u,v\in \mathbb{D}'$, we have 
\begin{align}\label{twoterm fe}
\mathscr{F}\left(x;u,v\right)+\mathscr{F}\left(\frac{1}{x};v,u\right)=-\log\left(1-u\right)\log\left(1-v\right).
\end{align}
\item For $ u,v, uv\in\mathbb{D}'$, we have
\begin{align}\label{three-term fe}
&\mathscr{F}(x;u,v)-\mathscr{F}\left(x+1;uv,v\right)-\mathscr{F}\left(\frac{x}{x+1};u,uv\right)\nonumber\\
&=\log(1-u)\log\left(1-uv\right)+\mathrm{Li}_2\left(u\right)-\mathrm{Li}_2\left(\frac{v}{v-1}\right)+2\mathrm{Li}_2\left(\frac{u}{u-1}\right)-\mathrm{Li}_2\left(\frac{u-v}{1-v}\right)\nonumber\\
&\quad-\left(\frac{1}{x+1}-\frac{1}{2}\right)\mathrm{Li}_2\left(uv\right)-\sum_{j=1}^2\left\{\mathrm{Li}_2\left(\frac{u+\sqrt{uv}e^{\pi ij}}{u-1}\right)
-\mathrm{Li}_2\left(\frac{v+\sqrt{uv}e^{\pi ij}}{v-1}\right)\right\}.
\end{align}
\item 
For $ u, v \in {\mathbb{D}_1}',$
we have
\begin{align}\label{six-term fe}
&\mathscr{F}(x;u,v)+\mathscr{F}(x;u^{-1},v^{-1})
-\mathscr{F}\left(x+1;uv,v\right)-\mathscr{F}\left(x+1;(uv)^{-1},v^{-1}\right)
\\
&-\mathscr{F}\left(\frac{x}{x+1};u,uv\right) 
 -\mathscr{F}\left(\frac{x}{x+1};u^{-1},(uv)^{-1}\right)\nonumber\\
&=\log(1-u)\log\left(1-uv\right)+\log(1-u^{-1})\log\left(1-(uv)^{-1}\right)+2\log\left(\frac{u}{u-1}\right)\log(1-u)\nonumber\\
&-\frac{1}{2}\log^2(-u)-\log\left(\frac{v}{v-1}\right)\log(1-v)+\left(\frac{1}{x+1}-\frac{1}{2}\right)\left(\frac{1}{2}\log^2(-uv)+\frac{\pi^2}{6}\right)-\mathrm{Li}_2\left(\frac{u-v}{1-v}\right)\nonumber\\
&-\mathrm{Li}_2\left(\frac{u-v}{u(1-v)}\right)-\sum_{j=1}^2\left\{\mathrm{Li}_2\left(\frac{u+\sqrt{uv}e^{\pi ij}}{u-1}\right)-\mathrm{Li}_2\left(\frac{v+\sqrt{uv}e^{\pi ij}}{v-1}\right)+\mathrm{Li}_2\left(\frac{1+\sqrt{uv^{-1}}e^{\pi ij}}{1-u}\right)\right.\nonumber\\
&\qquad\left.-\mathrm{Li}_2\left(\frac{1+\sqrt{vu^{-1}}e^{\pi ij}}{1-v}\right)\right\}.
\end{align}
\end{enumerate}
\end{theorem}

\begin{remark}
Let $\mathrm{Re}(x)>0$. It is easy to see that the two-term functional equation is valid for $u,v\in\mathbb{C}\backslash[1,\infty), (u,v)\neq(0,0)$. The three-term functional equation can also be derived for $u\in\mathbb{C}\backslash(1,\infty),\ u\neq0$ and $v\in\mathbb{C}\backslash[1,\infty),\ v\neq0$, namely, 
\begin{align}\label{fe general}
&\mathscr{F}(x;u,v)-\mathscr{F}(x+1;u v,v)-\mathscr{F}\left(\frac{x}{x+1};u,uv\right)\nonumber\\
&=\left(\frac{1}{2}-\frac{1}{x+1}\right)\mathrm{Li}_2(u v)-\mathscr{F}(2;uv,v)+\mathscr{F}(1;u,v)-\mathscr{F}\left(\frac{1}{2};u,uv\right).
\end{align}
See \eqref{general fe} below. We have to assume the said values of $u$ and $v$ in Theorem \ref{functional equations} to evaluate the constants involved on the right-hand side of \eqref{fe general} in terms of well-known functions. The six-term functional equation can also be obtained similarly using \eqref{fe general}. 
\end{remark}

\begin{remark}
If we introduce the function
\begin{align*}
\mathfrak{F}(x;u,v):=\mathscr{F}(x;u,v)+\frac{1}{2}\log(1-u)\log(1-v),
\end{align*}
then invoking \eqref{twoterm fe}, we obtain a simpler functional equation
\begin{align}\label{simpf}
\mathfrak{F}(x;u,v)+\mathfrak{F}\left(\frac{1}{x};v,u\right)=0.
\end{align}
It should be compared to \cite[Equation (2.7)]{raza}.
\end{remark}

The proposition given next states  connections between $F(x)$ and  $\mathscr{F}(x;u,v)$: 
\begin{proposition}  \label{connection $F$}
For $\mathrm{Re}(x)>0$, one has
\begin{align}
&(1)\lim_{\substack{u,v\to1}}\left\{\mathscr{F}(x;u,v)-\frac{1}{x}\mathrm{Li}_2(u)+\log(1-u)\left(\log(1-v)+\gamma+\log(x)\right)+\mathrm{Li}_s'(u)|_{s=1}\right\}=F(x).\nonumber\\
&(2)\ \mathscr{F}(x;1,-1)=F\left(\frac{x}{2}\right)-F(x)+\frac{\pi^2}{6x}.\label{hn to her}\\
&(3) \lim_{v\to1}\left\{\mathscr{F}(x;-1,v)+\log(2)\log(1-v)\right\}=F(2x)-F(x)-\frac{1}{2}\log(2)\log(2x^2)-\frac{\pi^2}{12x}.\label{vto1}
\end{align}
\end{proposition}

We end this subsection with the following remark.
\begin{remark}
In a forthcoming paper \cite{choiekumar}, we explore  a higher version of $\mathscr{F}(x;u,v)$ with its application in obtaining a ``Higher Kronecker limit formula" for real quadratic fields, which  extends  
the results in \cite{vz}.
\end{remark}


\medskip

\subsection{Period relations and cohomology}\label{co}

The expressions on the left-hand side of the two, three and six term functional equations are not arbitrary; instead, they satisfy the following rule. Define an action of GL$_2(\mathbb{R}) $ on $  \mathbb{C}' \times ( {\mathbb{D}_1}' \times  {\mathbb{D}_1}' )$, where  $ {\mathbb{D}_1}'$  is given in \eqref{set}, as 
\begin{equation*}
\begin{pmatrix}
a&b\\
c&d
\end{pmatrix}
\circ
\left[x,(u,v)\right]:=\left[\frac{ax+b}{cx+d},\left(u^av^b, u^cv^d\right)\right].
\end{equation*} 
For a function $f: \mathbb{C}' \times  ( {\mathbb{D}_1}' \times  {\mathbb{D}_1}' )\rightarrow \mathbb{C}$ define
\begin{eqnarray} \label{act}
 \left( \large {f\big{|}} \sm a& b\\ c& d\esm\right )(x, (u, v))  := f\left(\frac{ax+b}{cx+d}, \left(u^av^b, u^cv^d\right)\right).
\end{eqnarray}
One checks that
$$ \large (f|M |N \large)(x, (u,v))=
\large (f|MN\large )(x, (u,v)).$$

Relations involving $F(x), F(x+1)$, and $F\left(\frac{x}{x+1}\right)$ occur in several areas of mathematics, including period functions for modular forms,  cotangent functions, and double zeta functions, among others (see for instance \cite{choiezagier}, \cite{lewiszagier}, \cite{raza}, \cite{vz}). Zagier \cite{zagier} gave a beautifully detailed explanation of this topic, and we next follow his notations. Consider the modular group  $\Gamma:=\mathrm{SL}_2(\mathbb{Z}) $ and its generators $S=\sm 0 & -1\\ 1 & 0 \esm, U=\sm 1& -1\\ 1 & 0\esm$  and relations are:
$$\Gamma =\large< S, U\,  : \, S^4=U^6= I\large >.$$
Furthermore, let
$$T=\sm 1 & 1\\ 0 & 1\esm=-US,\ T'=\sm 1 & 0 \\ 1 & 1\esm  =U^2S=TST.$$
 
Let $\mathbb{Z}[\Gamma]$ be group ring of $\Gamma$ over   $\mathbb{Z}$ and let
 $V$ be a right $\mathbb{Z}[\Gamma]$-module.  
 The cohomology group (see \cite{B})
$$H^1(\Gamma, V)=Z^1(\Gamma,V)/ B^1(\Gamma, V)$$ where
$$Z^1(\Gamma,V)=\{f:\Gamma\rightarrow V :\, f_{\gamma_1\gamma_2}=f_{\gamma_1}|\gamma_2, \forall \gamma\in \Gamma\},$$
and
$$B^1(\Gamma, V):=\{f:\Gamma\rightarrow V :\,  f_{\gamma}=v_0|\gamma-v_0, \mbox{for some $v_0\in V$}\}.$$ 
 
Now take $V$ as the space of functions on $\mathbb{C}'\times ( {\mathbb{D}_1}' \times  {\mathbb{D}_1}')$ with the action of $\Gamma$ given by
$\gamma F=F|\gamma$ in (\ref{act}).
Then a $1$-cocycle $\phi$ is the same as a pair of functions $F=f_S$ and $G=f_{U}$ satisfying
\begin{eqnarray}\label{period}
& F|\Large{(}(\pm I)+(\pm S) \Large{)}=0 \nonumber \\
& G|\Large{(}(\pm I)+(\pm U)+(\pm U^2) \Large{)} =0.
\end{eqnarray}
 Since $U=TS$ so we have
$G=f_{T}|S+f_S=f_T|S + F$. So we may rewrite the relation of  $G$ in  (\ref{period}) 
as
\begin{eqnarray*}
&& F(x; u,v)+F\left(x; u^{-1},v^{-1}\right)-F(x+1;uv, v)-F\left(x+1;(uv)^{-1}, v^{-1}\right)\\
&& -F\left(\frac{x}{x+1}; u, uv\right)-F\left(\frac{x}{x+1}; u^{-1}, (uv)^{-1}\right)
=B(x; u,v), 
\end{eqnarray*}
where
$$B(x;u,v):=-f_T|S|\large{(}(\pm I)+(\pm U) +(\pm U^2) \Large{)}|S.$$





\medskip

\subsection{Special values of $\mathscr{F}\left(x;u,v\right)$ at rationals}

Our first result in this section provides the explicit evaluation of $\mathscr{F}\left(x;u,v\right)$, $x\in\mathbb{Q}$.
\begin{theorem}\label{f eval at rationals}
Let $(u,v)\in\mathbb{D}\times \mathbb{D}'$ and $p,q\in\mathbb{N}$. Function $\mathscr{F}\left(\frac{p}{q};u,v\right)$ can be evaluated as
\begin{align}\label{at rationals}
\mathscr{F}\left(\frac{p}{q};u,v\right)=\frac{q}{p}\mathrm{Li}_2(u)+\sum_{\alpha^p=1}\sum_{\beta^q=1}\left\{\mathrm{Li}_2\left(\frac{\beta v^{1/q}}{\beta v^{1/q}-1}\right)-\mathrm{Li}_2\left(\frac{\alpha u^{1/p}-\beta v^{1/q}}{1-\beta v^{1/q}}\right)\right\}.
\end{align}
\end{theorem}

Theorem \ref{f eval at rationals} gives the following two evaluations of $F(x;u,v)$: one with $p=n$ and $q=1$, and the other with $p=1$, $q=n$ and $u=1$.
\begin{corollary}\label{novikov funct eval}
For $n\in\mathbb{N}$ and $u,v\in\mathbb{D},\ v\neq1$, we have
\begin{align}\label{novikov funct eval at n}
\mathscr{F}(n;u,v)=n\ \mathrm{Li}_2\left(\frac{v}{v-1}\right)+\frac{1}{n}\mathrm{Li}_2(u)-\sum_{j=1}^n\mathrm{Li}_2\left(\frac{u^{1/n}e^{\frac{2\pi ij}{n}}-v}{1-v}\right),
\end{align}
and 
\begin{align}\label{novikov funct at u=1}
\mathscr{F}\left(\frac{1}{n};1,v\right)=-\frac{1}{n}\mathrm{Li}_2(v)-\frac{1}{2}\sum_{j=1}^n\log^2\left(1-v^{1/n}e^{\frac{2\pi ij}{n}}\right).
\end{align}
\end{corollary}



The two-term functional equation \eqref{fe2} implies that $\mathscr{F}\left(x;u,v\right)+\mathscr{F}\left( \frac{1}{x};v,u\right)$ is just the product of logarithms. Our next result shows that certain other combinations of $\mathscr{F}\left(x;u,v\right)$ can also be given in terms of logarithms only.
\begin{corollary}\label{combinations of novikov}
Let $n\in\mathbb{N}$ and $v\in\mathbb{D}'$. 
\begin{enumerate}
\item 
\begin{eqnarray*}\label{combinations of novikov 1}
&&\mathscr{F}\left(\frac{1}{n};1,v\right)+\mathscr{F}\left(\frac{1}{n};1,\frac{v}{v-1}\right) =\frac{1}{2n}\log^2(1-v)\\
&&-\frac{1}{2}\sum_{j=1}^n
\left(
\log^2\left(1-v^{1/n}e^{\frac{2\pi ij}{n}}\right)+\log^2\left(1-\left(\frac{v}{v-1}\right)^{1/n}e^{\frac{2\pi ij}{n}}\right)\right).
\end{eqnarray*}
\item Let   $0<\mathrm{Re}(v)<1$. Then,  
\begin{enumerate}
\item 
\begin{eqnarray}\label{combinations of novikov 2}
&&\mathscr{F}\left(\frac{1}{n};1,v\right)+\mathscr{F}\left(\frac{1}{n};1,1-v\right) =\frac{1}{n}\log(1-v)\log(v)-\frac{\pi^2}{6n}\nonumber\\
&&-\frac{1}{2}\sum_{j=1}^n\left(
\log^2\left(1-v^{1/n}e^{\frac{2\pi ij}{n}}\right)
 +\log^2\left(1-(1-v)^{1/n}e^{\frac{2\pi ij}{n}}\right)\right).
\end{eqnarray}
\item
\begin{eqnarray*}\label{combinations of novikov 4}
 &&\mathscr{F}\left(\frac{1}{n};1,\frac{v}{v-1}\right)-\mathscr{F}\left(\frac{1}{n};1,1-v\right) =\frac{\pi^2}{6n}+\frac{1}{2n}\log^2(1-v)-\frac{1}{n}\log(1-v)\log(v)\nonumber\\
&& +\frac{1}{2}\sum_{j=1}^n \left(
\log^2\left(1-(1-v)^{1/n}e^{\frac{2\pi ij}{n}}\right) - \log^2\left(1-\left(\frac{v}{v-1}\right)^{1/n}e^{\frac{2\pi ij}{n}}\right)\right).
\end{eqnarray*}
\end{enumerate}
\item For $\mathrm{Re}(v)<0$, we have
\begin{enumerate}
\item 
\begin{eqnarray*}\label{combinations of novikov 3}
&&\mathscr{F}\left(\frac{1}{n};1,v\right)-\mathscr{F}\left(\frac{1}{n};1,\frac{1}{1-v}\right)=\frac{\pi^2}{6n}-\frac{1}{2n}\log(1-v)\log\left(\frac{1-v}{v^2}\right)\nonumber\\
&&-\frac{1}{2}\sum_{j=1}^n \left(
\log^2\left(1-v^{1/n}e^{\frac{2\pi ij}{n}}\right)- \log^2\left(1-(1-v)^{-1/n}e^{\frac{2\pi ij}{n}}\right)\right).
\end{eqnarray*}

\item \begin{eqnarray} \label{combinations of novikov 5}
&&\mathscr{F}\left(\frac{1}{n};1,1-v\right)+\mathscr{F}\left(\frac{1}{n};1,\frac{1}{1-v}\right)\nonumber
=-\frac{\pi^2}{3n}+\frac{1}{2n}\log(1-v)\log\left(\frac{1-v}{v^2}\right)\nonumber\\
&& +\frac{1}{n}\log(1-v)\log(v) 
-\frac{1}{2}\sum_{j=1}^n\left(\log^2\left(1-(1-v)^{1/n}e^{\frac{2\pi ij}{n}}\right)-
\log^2\left(1-(1-v)^{-1/n}e^{\frac{2\pi ij}{n}}\right)\right).\nonumber
\end{eqnarray}\end{enumerate}
\end{enumerate}
\end{corollary}

The special case $v=1/2$ in equation \eqref{combinations of novikov 2} gives:
\begin{corollary}\label{elementary integral}
For $n\in\mathbb{N}$, we have
\begin{align}
\mathscr{F}\left(\frac{1}{n};1,\frac{1}{2}\right)=\int_0^1\frac{\log\left(1-t^{1/n}\right)}{2-t}dt=\frac{1}{2n}\log^22-\frac{\pi^2}{12n}-\frac{1}{2}\sum_{j=1}^n\log^2\left(1-2^{-1/n}e^{\frac{2\pi ij}{n}}\right).\nonumber
\end{align}
\end{corollary}

\section{Properties of  the function $J(x)$}\label{jx section}
Let $J(x)$ as defined in \eqref{j(x)}. Note that, from \eqref{novikov function}, as mentioned in the Introduction,
\begin{align}\label{j as special case}
 J(x)=-\mathscr{F}\left(x; -1, -1\right).
\end{align}

 
\subsection{Two-term functional equation  of  $J(x)$ }
As a special case of Theorem \ref{functional equations}, we see that $J(x)$ satisfies the following functional equation. 
\begin{theorem}\label{fe ja}
For $\mathrm{Re}(x)>0$ we have
\begin{align}\label{fe ja eqn}
J(x)+J\left(\frac{1}{x}\right)=\log^2(2).
\end{align}
\end{theorem}

\subsection{ Relation between $J(x)$ and $F(x)$}
We have the following relations. The first one is an analogue of \eqref{jxfx} and the other one can be considered as a ``three-term" functional equation for $J(x)$.

\begin{theorem}\label{3term thm}
For  $\mathrm{Re}(x)>0$, we have
\begin{align}
&(1)\ J(x)=F(x )-F\left(\frac{x-1}{2}\right)+F\left(\frac{x-1}{2x}\right)+\frac{\pi^2(2x+1)}{12x}+\frac{\gamma^2}{2}+\gamma_1+\frac{\log^2(2)}{2}+\mathrm{Li}_2\left(\frac{1}{x}\right).\label{jx repres}\\
&(2)\quad J(x)-J\left(x-1\right)+J\left(\frac{x}{x-1}\right)\nonumber\\
&\quad\quad=F\left(\frac{x}{2}\right)+F\left(\frac{x-1}{x}\right)-F\left(\frac{x-1}{2}\right)+2\mathrm{Li}_2\left(\frac{1}{x}\right)+\mathrm{Li}_2\left(\frac{x-1}{x}\right)-\frac{\pi^2(x^2+1)}{12x(x-1)}\nonumber\\
&\quad\quad\quad+\frac{1}{2}\gamma^2+\gamma_1+\frac{1}{2}\left(\log^2(2)+\log^2(2x)+\log^2\left(\frac{x}{x-1}\right)-\log^2(2(x-1))\right).\label{3tem fe for j(x)1 eqn}
\end{align}
Equivalently,
\begin{align}\label{3tem fe for j(x)2 eqn}
&J(x)-J\left(x-1\right)+J\left(\frac{x}{x-1}\right)\nonumber\\
&=\mathscr{F}(x;1,-1)-\mathscr{F}(x-1;1,-1)+\mathrm{Li}_2\left(\frac{1}{x}\right)+\mathrm{Li}_2\left(\frac{x-1}{x}\right)-\frac{\pi^2(2x+1)}{12x}\nonumber\\
&\qquad+\frac{1}{2}\left(\log^2(2)+\log^2(2x)+\log^2\left(\frac{x}{x-1}\right)-\log^2(2(x-1))\right).
\end{align}
 
\end{theorem}


\subsection{Special values of $J(x)$}
In this section, we  study the special values of $J(x)$ at units of real quadratic fields and at rationals. 

One of the important features of Theorem \ref{fe ja} is that it readily gives the explicit evaluation of $J(x)$ at $x=1/a$ whenever we know its value at $x=a$. As mentioned earlier, several special values of $J(x)$ are known, therefore, Theorem \ref{fe ja} gives the explicit evaluation at $1/x$. We next provide such examples. Radchenko and Zagier \cite[p.~244]{raza} defined the function
\begin{align}\label{cap J}
\mathcal{J}(x):=J(x)-\frac{\log^2(2)}{2}+\frac{\pi^2}{24}\left(x-\frac{1}{x}\right).
\end{align}
It is clear from \eqref{fe ja eqn} that
\begin{align}\label{cap J fe}
\mathcal{J}(x)+\mathcal{J}\left(\frac{1}{x}\right)=0.
\end{align}

The following result is provided in \cite[Theorem~5]{raza}.
\begin{theorem}\label{eval of J}
For all $n\geq1$, the value $\mathcal{J}(n+\sqrt{n^2\pm1})$ is a rational linear combination of $\zeta(2)$, $\log(2)\log(n+\sqrt{n^2\pm1})$, and the values $\varrho(\mathcal{B})$ for at most four narrow classes $\mathcal{B}$ of quadratic forms of discriminant $4^a(n^2\pm1)$, with $a=0,\pm1$.
\end{theorem}
The value $\varrho(\mathcal{B})$ in the result above is given by 
\begin{align*}
\varrho(\mathcal{B})=\lim_{s\to1}\left(D^{s/2}\zeta(\mathcal{B},s)-\frac{\log(\epsilon)}{s-1}\right),
\end{align*}
where $\mathcal{B}$ is an element of the narrow class group of the quadratic order $\mathcal{O}_D=\mathbb{Z}+\mathbb{Z}\frac{D+\sqrt{D}}{2}$ of discriminant $D>0$, and $\zeta(\mathcal{B},s)$ is the partial zeta function associated to $\mathcal{B}$, and $\epsilon>1$ is the smallest unit in $\mathcal{O}_D$ of norm $1$. Zagier \cite{zagier1975} evaluated $\varrho(\mathcal{B})$ in terms of the Herglotz function $F(x)$. For more details, see \cite{zagier1975} or \cite{raza}. The limit above is an example of what is known as a \emph{Kronecker limit formula}.

Theorem \ref{eval of J} along with \eqref{cap J fe} immediately yields:
\begin{corollary}\label{eval of J1}
Let $\mathcal{J}(x)$ be defined in \eqref{cap J}. For all $n\geq1$, the value $\mathcal{J}(\mp n\pm\sqrt{n^2\pm1})$ is a rational linear combination of $\zeta(2)$, $\log(2)\log(n+\sqrt{n^2\pm1})$, and the values $\varrho(\mathcal{B})$ for at most four narrow classes $\mathcal{B}$ of quadratic forms of discriminant $4^a(n^2\pm1)$, with $a=0,\pm1$.
\end{corollary}
We also note that with the help of Theorem \ref{eval of J}, authors \cite{raza} provided two tables containing the values of $\mathcal{J}(n+\sqrt{n^2\pm1})$. Therefore, their tables along with \eqref{cap J fe} give evaluations for $\mathcal{J}(\mp n\pm\sqrt{n^2\pm1})$.

Before giving the special values of $J(x)$ for $x$ naturals, we also mention that Muzaffar and Williams \cite[p.~60, Theorem 6]{muzwil} provided the evaluation of $J(2\ell+\sqrt{4\ell^2-1})$;  by using their result and \eqref{fe ja eqn} one can find the value of $J(2\ell-\sqrt{4\ell^2-1})$ as 
\begin{align*}
J(2\ell-\sqrt{4\ell^2-1})=\frac{1}{2}\log^2(2)-\frac{\sqrt{D}}{2}(\beta(D,G_1)-\beta(D,G_2))+\frac{\pi^2\sqrt{D}}{24}-\frac{1}{2}\log(2)\log(\epsilon).
\end{align*}
See \cite[p.~60, Theorem 6]{muzwil} for notations in the above expression. Several other evaluations of this kind can also be obtained by employing the results of \cite{muzwil} and Theorem \ref{fe ja} above. See Section \ref{tables} below.

Corollary \ref{novikov funct eval} above gives the evaluation of $J(x)$ at $x=n$ and $x=1/n$:
\begin{corollary}\label{J evaluation}
Let $n\in\mathbb{N}$. We have
\begin{align}\label{j(n) eval}
J\left(n\right)=\frac{\pi^2}{12}\left(\frac{1}{n}-n\right)+\frac{n}{2}\log^2(2)+\sum_{j=1}^n\mathrm{Li}_2\left(\frac{1}{2}\left(1+e^{\frac{\pi i}{n}\left(2j+1\right)}\right)\right).
\end{align}
Also, in view of \eqref{fe ja eqn}, we have
\begin{align}\label{j(1/n) eval}
J\left(\frac{1}{n}\right)=\frac{\pi^2}{12}\left(n-\frac{1}{n}\right)+\left(1-\frac{n}{2}\right)\log^2(2)-\sum_{j=1}^n\mathrm{Li}_2\left(\frac{1}{2}\left(1+e^{\frac{\pi i}{n}\left(2j+1\right)}\right)\right).
\end{align}
\end{corollary}

For even natural number arguments, Corollary \eqref{J evaluation} can be further simplified:
\begin{corollary}\label{J evaluation at even}
Let $m\in\mathbb{N}$. We have
\begin{align}\label{J evaluation at even eqn}
J(2m)=\frac{\pi^2}{48}\left(\frac{1}{m}-2m\right)+m\log^2(2)-\sum_{j=0}^{m-1}\log\left(\sin\left(\frac{\pi(2j+1)}{4m}\right)\right)\log\left(\cos\left(\frac{\pi(2j+1)}{4m}\right)\right),
\end{align}
and 
\begin{align}
J\left(\frac{1}{2m}\right)&=\frac{\pi^2}{48}\left(2m-\frac{1}{m}\right)+(1-m)\log^2(2)\nonumber\\
&\qquad+\sum_{j=0}^{m-1}\log\left(\sin\left(\frac{\pi(2j+1)}{4m}\right)\right)\log\left(\cos\left(\frac{\pi(2j+1)}{4m}\right)\right).\nonumber
\end{align}
\end{corollary}

Radchenko and Zagier \cite[Proposition 6]{raza} provided the following interesting connection between $\mathcal{J}(x)$ and the abelian Stark's conjecture. 
\begin{proposition}
Let $\epsilon>0$ be a quadratic unit with even trace. Assume the abelian Stark conjecture for the real quadratic field $\mathbb{Q}(\epsilon)$. Then $\mathcal{J}(\epsilon)$ is a rational linear combination of $\zeta(2), \log{2}\log{\epsilon},$  and  $2\times 2$ determinants of algebraic units in the narrow right class field of the quadratic order $\mathbb{Z}[2\epsilon]$.
\end{proposition}

Employing \eqref{cap J fe} and the above proposition, we can conclude that
\begin{proposition}
Let $\epsilon>0$ be a quadratic unit with even trace. Assume the abelian Stark conjecture for the real quadratic field $\mathbb{Q}(\epsilon)$. Then $\mathcal{J}(1/\epsilon)$ is a rational linear combination of $\zeta(2), \log{2}\log{\epsilon},$  and  $2\times 2$ determinants of algebraic units in the narrow right class field of the quadratic order $\mathbb{Z}[2\epsilon]$.
\end{proposition}


\section{Properties of $T(x)$}\label{tx section}

The following integral, also defined in \eqref{t(x) defn} above, 
\begin{align}\label{T(x) defn}
T(x):=\int_0^1\frac{\tan^{-1}\left(t^x\right)}{1+t^2}dt, \quad \mathrm{Re}(x)>0.
\end{align}
appears  in the Kronecker limit formula obtained by Muzaffar and Williams \cite{muzwil}. They evaluated it at various values of $x$ \cite[Theorem 1.10]{muzwil}. Also see  \cite[p.~271, Exercise (60)]{cohen} and \cite[p.~169]{br}.

Our first result of this section shows that $T(x)$ can be written as a linear combination of the Herglotz-Zagier-Novikov function $\mathscr{F}\left(x;u,v\right)$.
\begin{theorem}\label{representation of T(x)}
Let $\mathrm{Re}(x)>0$. Then
\begin{align}\label{representation of T(x) eqn}
T(x)=\frac{1}{4}\left\{\mathscr{F}\left(x;i,i\right)+\mathscr{F}\left(x;-i,-i\right)-\mathscr{F}\left(x;i,-i\right)-\mathscr{F}\left(x;-i,i\right)\right\}.
\end{align}
\end{theorem}
The above result can be considered as an analogue of \eqref{j as special case}.

\subsection{ Relation between $T(x), J(x)$ and $\mathscr{F}\left(x;u,v\right)$}
The following identity relates the functions $T(x),\ J(x)$ and $\mathscr{F}\left(x;u,v\right)$.
\begin{theorem}\label{relation b/w j t f}
For $\mathrm{Re}(x)>0$,  the relation 
\begin{align}\label{relation b/w j t f eqn}
4\ T(x)+J(x)+2\mathscr{F}\left(x;i,-i\right)+2\mathscr{F}\left(x;-i,i\right)=0,
\end{align}
holds true.
\end{theorem}

\subsection{Two-term functional equation for $T(x)$}
The function $T(x)$ also satisfies the following functional equation similar to that of $J(x)$.
\begin{theorem}\label{fe for T(x)}
Let $\mathrm{Re}(x)>0$. Then, we have
\begin{align}\label{fe for T(x) eqn}
T(x)+T\left(\frac{1}{x}\right)
&=\frac{\pi^2}{16}.
\end{align}
\end{theorem}

\begin{remark}
If we define the function
\begin{align*}
\mathcal{T}(x):=T(x)-\frac{\pi^2}{32},
\end{align*}
then using \eqref{fe for T(x) eqn}, we find that
\begin{align*}
\mathcal{T}(x)+\mathcal{T}\left(\frac{1}{x}\right)=0.
\end{align*}
It should be compared with \eqref{simpf} or \eqref{cap J fe}.
\end{remark}

\subsection{ Special values of $T(x)$ }

Our next result evaluates the function $T(x)$ at natural numbers and their reciprocals in terms of dilogarithm functions.
\begin{corollary}\label{T(n) eval thm}
For $n\in\mathbb{N}$,
\begin{align}\label{T(n) eval}
T(n)&=\frac{1}{4}\sum_{j=1}^n\left\{\mathrm{Li}_2\left(\frac{1+i}{2}\left(1- e^{\frac{\pi i}{2n}(4j+n+1)}\right)\right)+\mathrm{Li}_2\left(\frac{1-i}{2}\left(1+e^{\frac{\pi i}{2n}(4j+n-1)}\right)\right)\right.\nonumber\\
&\quad\left.-\mathrm{Li}_2\left(\frac{1-i}{2}\left(1+e^{\frac{\pi i}{2n}(4j+n+1)}\right)\right)+\mathrm{Li}_2\left(\frac{1+i}{2}\left(1-e^{\frac{\pi i}{2n}(4j+n-1)}  \right)\right)\right\}.
\end{align}
Also, by virtue of \eqref{fe for T(x) eqn},
\begin{align}\label{T(1/n) eval}
T\left(\frac{1}{n}\right)&=\frac{\pi^2}{16}-\frac{1}{4}\sum_{j=1}^n\left\{\mathrm{Li}_2\left(\frac{1+i}{2}\left(1-e^{\frac{\pi i}{2n}(4j+n+1)}\right)\right)+\mathrm{Li}_2\left(\frac{1-i}{2}\left(1+e^{\frac{\pi i}{2n}(4j+n-1)}\right)\right)\right.\nonumber\\
&\left.\quad-\mathrm{Li}_2\left(\frac{1-i}{2}\left(1+e^{\frac{\pi i}{2n}(4j+n+1)}\right) \right)+\mathrm{Li}_2\left(\frac{1+i}{2}\left(1- e^{\frac{\pi i}{2n}(4j+n-1)} \right)\right)\right\}.
\end{align}
\end{corollary}


We end this section by giving an application of Theorem \ref{fe for T(x)} in evaluating $T(x)$. The result below is derived in \cite[p.~64, Theorem 10]{muzwil}:
\begin{align*}
T(3+\sqrt{8})=\int_0^1\frac{\mathrm{tan}^{-1}\left(t^{3+\sqrt{8}}\right)}{1+t^2}dt=\frac{1}{16}\log(2)\log(3+\sqrt{8}).
\end{align*}
It is straightforward to conclude from \eqref{fe for T(x) eqn} and the result above that
\begin{align*}
T(3-\sqrt{8})=\int_0^1\frac{\mathrm{tan}^{-1}\left(t^{3-\sqrt{8}}\right)}{1+t^2}dt=\frac{1}{16}\left(\pi^2-\log(2)\log(3+\sqrt{8})\right),
\end{align*}
which is a new evaluation, to the best of our knowledge. Similar new special values for $J(x)$ and $T(x)$ can be obtained from \cite[p.~64, Theorem 10]{muzwil} and our Theorem \ref{fe ja eqn} and Theorem \ref{fe for T(x)}, see Section \ref{tables} below.

\section{Proofs}

\subsection{Proof of the functional equations for $\mathscr{F}\left(x;u,v\right)$}\label{fe of hn}
We first prove \eqref{twoterm fe}.

Making the change of variable $t=e^{-u}$ in \eqref{novikov function}, one obtains 
\begin{align}
\mathscr{F}\left(x;u,v\right)=\int_0^\infty\frac{\log\left(1-ue^{-xt}\right)}{v^{-1}e^{t}-1}dt.\nonumber
\end{align}
We differentiate both sides of the above equation with respect to $x$ to get
\begin{align}\label{der of rho}
\mathscr{F}'\left(x;u,v\right)
&=\int_0^\infty\frac{t}{\left(v^{-1}e^t-1\right)\left(u^{-1}e^{xt}-1\right)}dt.
\end{align}
Employing the change of variable $t\to t/x$ in the integral on the right-hand side of \eqref{der of rho}, we obtain
\begin{align}
\mathscr{F}'\left(x;u,v\right)
&=\frac{1}{x^2}\int_0^\infty\frac{t}{\left(v^{-1}e^{t/x}-1\right)\left(u^{-1}e^{t}-1\right)}dt\nonumber\\
&=\frac{1}{x^2}\mathscr{F}'\left(\frac{1}{x};v,u\right),\nonumber
\end{align}
which follows upon using \eqref{der of rho}.
Integrating the above equation, we deduce that
\begin{align}\label{with c}
\mathscr{F}\left(x;u,v\right)+\mathscr{F}\left(\frac{1}{x};v,u\right)=c_1.
\end{align}
The integrating constant $c_1$ can be evaluated to $\mathscr{F}\left(1;u,v\right)+\mathscr{F}\left(1;v,u\right)$ by taking $x=1$ in \eqref{with c}. Invoking \eqref{novikov funct eval at n}  twice to get
\begin{align}\label{evaluation}
c_1=\mathscr{F}\left(1;u,v\right)+\mathscr{F}\left(1;v,u\right)&=\mathrm{Li}_2(u)+\mathrm{Li}_2\left(\frac{u}{u-1}\right)+\mathrm{Li}_2(v)+\mathrm{Li}_2\left(\frac{v}{v-1}\right)\nonumber\\
&\quad-\mathrm{Li}_2\left(\frac{u-v}{1-v}\right)-\mathrm{Li}_2\left(\frac{v-u}{1-u}\right)\nonumber\\
&=-\frac{1}{2}\log^2(1-u)-\frac{1}{2}\log^2(1-v)+\frac{1}{2}\log^2\left(\frac{1-u}{1-v}\right)\nonumber\\
&=-\log(1-u)\log(1-v),
\end{align}
where in the penultimate step we used \cite[p.~610, Formula 25.12.3]{nist}
\begin{align}\label{polylog x and x-1}
\mathrm{Li}_2\left(\frac{x}{x-1}\right)+\mathrm{Li}_2(x)=-\frac{1}{2}\log^2(1-x)\quad (x\in\mathbb{C}\backslash[1,\infty)).
\end{align}
Thus, we obtain the two-term functional equation \eqref{twoterm fe} after substituting \eqref{evaluation} in \eqref{with c}.

We next prove the three-term functional equation \eqref{three-term fe}.

Firstly, note that
\begin{align}
&\frac{v}{\left(e^{xt}-u\right)\left(e^{t}-v\right)}=\frac{1}{\left(e^{(x-1)t}-uv^{-1}\right)\left(e^{t}-v\right)}-\frac{uv^{-1}}{\left(e^{(x-1)t}-uv^{-1}\right)\left(e^{xt}-u\right)}-\frac{1}{e^{xt}-u}.\nonumber
\end{align}
It implies that
\begin{align*}
&\frac{1}{\left(v^{-1}e^t-1\right)\left(u^{-1}e^{xt}-1\right)}\\
&=\frac{1}{\left((uv^{-1})^{-1}e^{(x-1)t}-1\right)\left(v^{-1}e^t-1\right)}-\frac{1}{\left((uv^{-1})^{-1}e^{(x-1)t}-1\right)\left(u^{-1}e^{xt}-1\right)}-\frac{1}{u^{-1}e^{xt}-1}.
\end{align*}
We multiply the above equation by $t$ and integrate and then use \eqref{der of rho}, so as to obtain
\begin{align}\label{before int}
\mathscr{F}'(x;u,v)-\mathscr{F}'(x-1;uv^{-1},v)+\frac{1}{x^2}\mathscr{F}'\left(\frac{x-1}{x};uv^{-1},u\right)=-\frac{1}{x^2}\mathrm{Li}_2(u),
\end{align}
where on the right-hand side we used  \cite[p.~611, Formula 25.12.11]{nist}
\begin{align*}
\mathrm{Li}_s(z)=\frac{z}{\Gamma(s)}\int_0^\infty\frac{t^{s-1}}{e^{t}-z}dt,\quad (\mathrm{Re}(s)>0, |\arg(1-z)|<\pi).
\end{align*}
Integrating \eqref{before int}, we arrive at
\begin{align}\label{constants}
\mathscr{F}(x;u,v)-\mathscr{F}(x-1;uv^{-1},v)+\mathscr{F}\left(\frac{x-1}{x};uv^{-1},u\right)=\frac{1}{x}\mathrm{Li}_2(u)+c_2.
\end{align}
Let $x=2$ in \eqref{constants} to evaluate the constant $c_2$ as
\begin{align}\label{inv}
c_2=\mathscr{F}(2;u,v)-\mathscr{F}(1;uv^{-1},v)+\mathscr{F}\left(\frac{1}{2};uv^{-1},u\right)-\frac{1}{2}\mathrm{Li}_2(u).
\end{align}
Combining \eqref{constants} and \eqref{inv}, we are led to
\begin{align}\label{general fe}
&\mathscr{F}(x;u,v)-\mathscr{F}(x-1;uv^{-1},v)+\mathscr{F}\left(\frac{x-1}{x};uv^{-1},u\right)\nonumber\\
&=\left(\frac{1}{x}-\frac{1}{2}\right)\mathrm{Li}_2(u)+\mathscr{F}(2;u,v)-\mathscr{F}(1;uv^{-1},v)+\mathscr{F}\left(\frac{1}{2};uv^{-1},u\right).
\end{align}
Note that the above equation is valid for $u\in\mathbb{C}\backslash[1,\infty),\ u\neq0$ and $v\in\mathbb{C}\backslash[1,\infty),\ v\neq0$. Further restrictions on the parameters in the hypothesis arise from the application of Corollary \ref{novikov funct eval at n}, which allows us to determine the value of $c_2$ in terms of well-known functions instead of those appearing in \eqref{inv}.
 
Invoking \eqref{twoterm fe} in \eqref{inv}, one can see that 
\begin{align}
c_2=\mathscr{F}(2;u,v)-\mathscr{F}(1;uv^{-1},v)-\mathscr{F}\left(2;u,uv^{-1}\right)-\log(1-uv^{-1})\log(1-u)-\frac{1}{2}\mathrm{Li}_2(u).\nonumber
\end{align}
Upon employing Corollary \ref{novikov funct eval at n} multiple times in the above expression, $c_2$ can be evaluated as
\begin{align}\label{c2}
c_2&=\mathrm{Li}_2\left(\frac{v}{v-1}\right)-\mathrm{Li}_2\left(\frac{u}{v}\right)-2\mathrm{Li}_2\left(\frac{u}{u-v}\right)+\mathrm{Li}_2\left(\frac{u-v^2}{v-v^2}\right)-\log(1-u)\log\left(1-uv^{-1}\right)\nonumber\\
&\quad+\sum_{j=1}^2\mathrm{Li}_2\left(\frac{u+ve^{\pi ij}\sqrt{u}}{u-v}\right)
-\sum_{j=1}^2\mathrm{Li}_2\left(\frac{v+\sqrt{u}e^{\pi ij}}{v-1}\right)-\frac{1}{2}\mathrm{Li}_2\left(u\right).
\end{align}
Substituting the value of $c_2$ from \eqref{c2} in \eqref{constants} and then replacing $x$ by $x+1$ and $u$ by $uv$ in the resulting expression, we get the three-term functional equation \eqref{three-term fe}. 

To prove \eqref{six-term fe}, we first replace $u$ and $v$ by $u^{-1}$ and $v^{-1}$ respectively in \eqref{three-term fe} and add the resulting expression to \eqref{three-term fe} and then simplify the right-hand side by using certain identities of dilogarithm function.
\qed

\begin{proof}[Theorem \ref{sum formula for Jx}][]
Definition \eqref{novikov function} implies that
\begin{align*}
\mathscr{F}\left(2x;u^2,v\right)&=\int_0^1\frac{\log\left(1-t^{2x}u^2\right)}{v^{-1}-t}dt\nonumber\\
&=\int_0^1\frac{\log\left(1-t^{x}u\right)+\log\left(1+t^{x}u\right)}{v^{-1}-t}dt.
\end{align*}
Equation \eqref{sum formula for Jx eqn1} is now straightforward upon using \eqref{novikov function} in the above expression.

To prove \eqref{sum formula for Jx eqn2}, we break the integral in \eqref{novikov function} as 
\begin{align*}
\mathscr{F}\left(\frac{x}{2};u,v^2\right)&=\frac{1}{2}\int_0^\infty\frac{\log\left(1-e^{-\frac{xt}{2}}u\right)}{e^{t/2}v^{-1}-1}dt-\frac{1}{2}\int_0^\infty\frac{\log\left(1-e^{-\frac{xt}{2}}u\right)}{e^{t/2}v^{-1}+1}dt.
\end{align*}
We arrive at \eqref{sum formula for Jx eqn2} after  making the change of variable $t=2y$ and then using \eqref{novikov function}.
\end{proof}

We next prove the relation between $\mathscr{F}\left(x;u,v\right)$ and $F(x)$ given in \eqref{hn to her}.
\begin{proof}[Proposition \textup{\ref{connection $F$}}][]
Definition of $\mathscr{F}(x;u,v)$ implies that
\begin{align}\label{partial fn}
\mathscr{F}(x;u,v)&=\int_0^1 \log\left(1-ut^x\right)\left(\frac{t}{v^{-1}-t}+1-1\right)\frac{dt}{t}\nonumber\\
&=\int_0^1 \frac{\log\left(1-ut^x\right)}{1-vt}\frac{dt}{t}+\frac{1}{x}\mathrm{Li}_2(u).
\end{align}
Assume $|u|<1$ and $|v|<1$ and expand the logarithm and $1/(1-vt)$ as their series expansions around $t=0$ and then interchange the order of summation and integration so that
\begin{align}\label{before exp}
\mathscr{F}(x;u,v)&=-\sum_{n=1}^\infty\sum_{m=0}^\infty\frac{u^nv^m}{n}\int_0^1t^{m+nx-1}dt+\frac{1}{x}\mathrm{Li}_2(u)\nonumber\\
&=-\sum_{n=1}^\infty\sum_{m=0}^\infty\frac{u^nv^m}{n(m+nx)}+\frac{1}{x}\mathrm{Li}_2(u).
\end{align}
Next, utilizing the facts $\sum_{m=0}^\infty\sum_{n=1}^\infty\frac{u^nv^{m+1}}{n(m+1)}=\log(1-u)\log(1-v)$, $\gamma\sum_{n=1}^\infty\frac{u^n}{n}=-\gamma\log(1-u)$, $\sum_{n=1}^\infty\frac{u^n\log(nx)}{n}=-\log(x)\log(1-u)-\mathrm{Li}_s'(u)|_{s=1}$, one can evaluate
\begin{align}\label{exp}
&-\sum_{m=0}^\infty\sum_{n=1}^\infty\frac{u^nv^m}{n(m+nx)}+\log(1-u)\log(1-v)+\gamma\log(1-u)+\log(x)\log(1-u)+\mathrm{Li}_s'(u)|_{s=1}\nonumber\\
&=\sum_{n=1}^\infty\frac{u^n}{n}\sum_{m=0}^\infty\left(\frac{v^{m+1}}{m+1}-\frac{v^{m}}{m+nx}\right)+\gamma\log(1-u)+\log(x)\log(1-u)+\mathrm{Li}_s'(u)|_{s=1}\nonumber\\
&=\sum_{n=1}^\infty\frac{u^n}{n}\left\{\sum_{m=0}^\infty\left(\frac{v^{m+1}}{m+1}-\frac{v^{m}}{m+nx}\right)-\gamma\right\}+\log(x)\log(1-u)+\mathrm{Li}_s'(u)|_{s=1}\nonumber\\
&=\sum_{n=1}^\infty\frac{u^n}{n}\left\{\sum_{m=0}^\infty\left(\frac{v^{m+1}}{m+1}-\frac{v^{m}}{m+nx}\right)-\gamma-\log(nx)\right\}.
\end{align}
Equations \eqref{before exp} and \eqref{exp}  together yield
\begin{align}
&\mathscr{F}(x;u,v)-\frac{1}{x}\mathrm{Li}_2(u)+\log(1-u)\log(1-v)+\gamma\log(1-u)+\log(x)\log(1-u)+\mathrm{Li}_s'(u)|_{s=1}\nonumber\\
&=\sum_{n=1}^\infty\frac{u^n}{n}\left\{\sum_{m=0}^\infty\left(\frac{v^{m+1}}{m+1}-\frac{v^{m}}{m+nx}\right)-\gamma-\log(nx)\right\}.
\end{align}
Taking limit $u\to1$ and $v\to1$ on both sides of the above equation, we conclude that
\begin{align*}
&\lim_{\substack{u\to1\\v\to1}}\bigg\{\mathscr{F}(x;u,v)-\frac{1}{x}\mathrm{Li}_2(u)+\log(1-u)\log(1-v)+\gamma\log(1-u)+\log(x)\log(1-u)\nonumber\\
&\qquad+\mathrm{Li}_s'(u)|_{s=1}\bigg\}
=\sum_{n=1}^\infty\frac{1}{n}\left\{\sum_{m=0}^\infty\left(\frac{1}{m+1}-\frac{1}{m+nx}\right)-\gamma-\log(nx)\right\}.
\end{align*}
Now using the series expansion $\psi(x)=\sum_{m=0}^\infty\left(\frac{1}{m+1}-\frac{1}{m+x}\right)-\gamma$ in the above equation and then employing \eqref{herglotz defn ser}, we complete the proof of the first part.

We next prove the second part. From the definition of $\mathscr{F}\left(x;u,v\right)$ from \eqref{novikov function}, we observe that
\begin{align}
\mathscr{F}\left(x;1,-1\right)&=-\int_0^\infty\frac{\log\left(1-e^{-xt}\right)}{e^t+1}dt\nonumber\\
&=-\int_0^\infty\left(\frac{1}{e^t-1}-\frac{2}{e^{2t}-1}\right)\log\left(1-e^{-xt}\right)dt\nonumber\\
&=-\int_0^\infty\left(\frac{1}{e^t-1}-\frac{1}{t}\right)\log\left(1-e^{-xt}\right)dt+\int_0^\infty\left(\frac{2}{e^{2t}-1}-\frac{1}{t}\right)\log\left(1-e^{-xt}\right)dt.\nonumber
\end{align}
We first make the change of variable $t\to t/2$ in the second integral on the right-hand side of the above equation and then add and subtract $1$ to the integrand of the both integrals so as to obtain
\begin{align}
\mathscr{F}\left(x;1,-1\right)&=-\int_0^\infty\left(\frac{1}{1-e^{-t}}-\frac{1}{t}\right)\log\left(1-e^{-xt}\right)dt+\int_0^\infty\left(\frac{1}{1-e^{-t}}-\frac{1}{t}\right)\log\left(1-e^{-\frac{xt}{2}}\right)dt\nonumber\\
&\qquad-\int_0^\infty\log\left(1-e^{-xt}\right)dt+\int_0^\infty\log\left(1-e^{-xt/2}\right)dt\nonumber\\
&=-F(x)+F\left(\frac{x}{2}\right)+\frac{\pi^2}{6x},\nonumber
\end{align}
where in the last step we used the definition of $F(x)$ given in \eqref{herglotz defn} and the identity  $\int_0^\infty\log\left(1-e^{-xt}\right)dt\newline=\frac{\pi^2}{6x}$.
This completes the proof of the relation \eqref{hn to her}.

The proof of \eqref{vto1} proceeds  as follows. Taking $u=-1$ in \eqref{twoterm fe} and rearranging the terms and then letting the limit $v\to1$ gives
\begin{align}
\lim_{v\to1}\left\{\mathscr{F}(x;-1,v)+\log(2)\log(1-v)\right\}=\mathscr{F}\left(\frac{1}{x};1,-1\right).
\end{align}
Utilizing \eqref{hn to her} with replacing $x$ by $1/x$ and then invoking \eqref{fe2} twice and simplifying, we arrive at \eqref{vto1}.
\end{proof}

\begin{proof}[Proposition \textup{\ref{analytic continuation 1}}][]
Expanding the logarithm and $1/(1-vt)$ as their series expansions around $t=0$ and then interchange the order of summation and integration, we have
\begin{align*}
\mathscr{F}(x;u,v)&=-\sum_{m=1}^\infty\sum_{n=0}^\infty\frac{u^mv^{n+1}}{m}\int_0^1 t^{mx+n}dt\\
&=-\sum_{m=1}^\infty\sum_{n=1}^\infty\frac{u^mv^{n}}{m(mx+n)}.
\end{align*}
It is clear that the series on the right-hand side of the above equation is uniformly convergent when $x\in\mathbb{C}'$ owing to the fact that it has decay for the said values of $u$ and $v$. Therefore, the right-hand side serves as an analytic continuation for the function $\mathscr{F}(x;u,v)$.
\end{proof}

\subsection{Evaluation of $\mathscr{F}(x;u,v)$ at rationals}

We start with the following lemma.

\begin{lemma}\label{int eval}
Let $\alpha,\beta\in\mathbb{D}'$, where $\mathbb{D}'$ is defined in \eqref{set}. Then
\begin{align}\label{int eval eqn}
\int_0^1\frac{\log(1-\alpha t)}{t(1-\beta t)}dt=\mathrm{Li}_2\left(\frac{\beta}{\beta-1}\right)-\mathrm{Li}_2\left(\frac{\alpha-\beta}{1-\beta}\right).
\end{align}
\end{lemma}
\begin{proof}
Let us denote
\begin{align}\label{fab}
f(\alpha,\beta):=\int_0^1\frac{\log(1-\alpha t)}{t(1-\beta t)}dt.
\end{align}
Differentiating with respect to $\alpha$ so as to obtain
\begin{align}\label{da1}
\frac{d}{d\alpha}f(\alpha,\beta)&=-\int_0^1\frac{1}{(1-\alpha t)(1-\beta t)}dt\nonumber\\
&=\frac{1}{\alpha-\beta}\int_0^1\left(\frac{\beta}{1-\beta t}-\frac{\alpha}{1-\alpha t}\right)dt=\frac{1}{\alpha-\beta}(\log\left(1-\alpha\right)-\log\left(1-\beta\right))\nonumber\\
&=\frac{1}{\alpha-\beta}\log\left(\frac{1-\alpha}{1-\beta}\right),
\end{align}
as $\alpha,\beta\in\mathbb{D}'$. Next, it is easy to see that
\begin{align}\label{der of li2}
\frac{d}{dz}\mathrm{Li}_2(z)=-\frac{1}{z}\log(1-z).
\end{align}
Using \eqref{der of li2}, we compute
\begin{align*}
\frac{d}{d\alpha}\mathrm{Li}_2\left(\frac{\alpha-\beta}{1-\beta}\right)=-\frac{1}{\alpha-\beta}\log\left(\frac{1-\alpha}{1-\beta}\right).
\end{align*}
Substituting the above evaluation in \eqref{da1}, we obtain
\begin{align}
\frac{d}{d\alpha}f(\alpha,\beta)&=-\frac{d}{d\alpha}\mathrm{Li}_2\left(\frac{\alpha-\beta}{1-\beta}\right).\nonumber
\end{align}
Integrating both sides so that
\begin{align}\label{with c1}
f(\alpha,\beta)=-\mathrm{Li}_2\left(\frac{\alpha-\beta}{1-\beta}\right)+c.
\end{align}
If we let $\alpha=1$ in \eqref{with c1}, then $c=\frac{\pi^2}{6}+f(1,\beta)$ by using the fact $\mathrm{Li}_2(1)=\pi^2/6$. Therefore, we have
\begin{align}\label{aim}
f(\alpha,\beta)=-\mathrm{Li}_2\left(\frac{\alpha-\beta}{1-\beta}\right)+\frac{\pi^2}{6}+f(1,\beta).
\end{align}
Our next aim is to evaluate $f(1,\beta)$. To that end, using \eqref{fab}, one has
\begin{align}\label{da2}
f(1,\beta)&=\beta\int_0^1\frac{\log(1-t)}{1-\beta t}dt+\int_0^1\frac{\log(1-t)}{t}dt.
\end{align}
Using the integral evaluation
$\int_0^1\frac{\log(1-t)}{t}dt=-\frac{\pi^2}{6}$
in \eqref{da2}, we have
\begin{align}\label{da3}
f(1,\beta)&=\beta\int_0^1\frac{\log(1-t)}{1-\beta t}dt-\frac{\pi^2}{6}.
\end{align}
Equation \eqref{der of li2} implies
\begin{align}
\frac{d}{dt}\mathrm{Li}_2\left(\frac{1-\beta t}{1-\beta}\right)=\frac{\beta}{1-\beta t}\log\left(\frac{\beta}{\beta-1}\right)+\beta\frac{\log(1-t)}{1-\beta t}.\nonumber
\end{align}
It gives
\begin{align}\label{betaint}
\beta\int_0^1\frac{\log(1-t)}{1-\beta t}dt&=\int_0^1d\mathrm{Li}_2\left(\frac{1-\beta t}{1-\beta}\right)-\beta\log\left(\frac{\beta}{\beta-1}\right)\int_0^1\frac{1}{1-\beta t}dt\nonumber\\
&=\frac{\pi^2}{6}-\mathrm{Li}_2\left(\frac{1}{1-\beta}\right)+\log\left(\frac{\beta}{\beta-1}\right)\log(1-\beta)\nonumber\\
&=\mathrm{Li}_2\left(\frac{\beta}{\beta-1}\right),
\end{align}
where we used the Euler's relation \cite[p.~5, Equation (1.11)]{lewin}
\begin{align}\label{euler}
\mathrm{Li}_2(z)=-\mathrm{Li}_2(1-z)-\log(1-z)\log(z)+\frac{\pi^2}{6}.
\end{align}
Substitute value from \eqref{betaint} in \eqref{da3} to evaluate $f(1,\beta)$ as
\begin{align*}
f(1,\beta)&=\mathrm{Li}_2\left(\frac{\beta}{\beta-1}\right)-\frac{\pi^2}{6}.
\end{align*}
Now lemma readily follows upon inserting the value of $f(1,\beta)$ from the above equation into \eqref{aim}.
\end{proof}

\begin{remark}
Lemma \ref{int eval} can be found in \cite{raza} where the authors indicated the proof without giving the details. We here filled its details to make the paper self-contained.
\end{remark}

\begin{proof}[Theorem \textup{\ref{f eval at rationals}}][]
Letting $x=p/q$ in \eqref{partial fn} and making the change of variable $t$ by $t^q$ to get
\begin{align}\label{rational1}
\mathscr{F}\left(\frac{p}{q};u,v\right)
&=\int_0^1 \log\left(1-ut^p\right)\frac{q}{1-vt^q}\frac{dt}{t}+\frac{q}{p}\mathrm{Li}_2(u).
\end{align} 
Using the relations
\begin{align}\label{distribution}
\log(1-y^p)=\sum_{\alpha^p=1}\log(1-\alpha y),
\end{align}
with $y=u^{1/p}t$, and
\begin{align*}
\frac{q}{1-y^q}=\sum_{\beta^q=1}\frac{1}{1-\beta y},
\end{align*}
with $y=v^{1/q}t$, in \eqref{rational1}, we obtain
\begin{align}
\mathscr{F}\left(\frac{p}{q};u,v\right)&=\sum_{\alpha^p=1}\sum_{\beta^q=1}f(\alpha,\beta,u,v)+\frac{q}{p}\mathrm{Li}_2(u),\nonumber
\end{align}
where
\begin{align}
f(\alpha,\beta,u,v):=\int_0^1\frac{\log(1-u^{1/p}\alpha t)}{t(1-v^{1/q}\beta t)}dt.\nonumber
\end{align}
Invoking Lemma \ref{int eval}, we find that
\begin{align}
f(\alpha,\beta,u,v)=\mathrm{Li}_2\left(\frac{v^{1/q}\beta}{v^{1/q}\beta-1}\right)-\mathrm{Li}_2\left(\frac{u^{1/p}\alpha-v^{1/q}\beta}{1-v^{1/q}\beta}\right).\nonumber
\end{align}
This proves the theorem.
\end{proof}

\begin{proof}[Corollary \textup{\ref{novikov funct eval}}][]
Letting $p=n$ and $q=1$ and making the change of variable $\alpha=e^{2\pi ij/n},\ 1\leq j\leq n$ in \eqref{at rationals} readily gives the result.

We next prove \eqref{novikov funct at u=1}. Let $p=1,\ q=n$ and $u=1$ in \eqref{at rationals} so that
\begin{align}\label{1v}
\mathscr{F}\left(\frac{1}{n};1,v\right)&=n\mathrm{Li}_2(1)+\sum_{\beta^n=1}\left\{\mathrm{Li}_2\left(\frac{\beta v^{1/n}}{\beta v^{1/n}-1}\right)-\mathrm{Li}_2(1)\right\}\nonumber\\
&=\sum_{\beta^n=1}\left\{\mathrm{Li}_2\left(\frac{\beta v^{1/n}}{\beta v^{1/n}-1}\right)\right\}\nonumber\\
&=\sum_{\beta^n=1}\left\{\mathrm{Li}_2\left(\frac{\beta v^{1/n}}{\beta v^{1/n}-1}\right)+\mathrm{Li}_2\left(\beta v^{1/n}\right)\right\}-\sum_{\beta^n=1}\mathrm{Li}_2\left(\beta v^{1/n}\right).
\end{align}
We now make use of the relation \eqref{polylog x and x-1} in \eqref{1v} to see that
\begin{align}\label{1v1}
\mathscr{F}\left(\frac{1}{n};1,v\right)&=-\frac{1}{2}\sum_{\beta^n=1}\log^2\left(1-\beta v^{1/n}\right)-\sum_{\beta^n=1}\mathrm{Li}_2\left(\beta v^{1/n}\right).
\end{align}
Using \cite[p.~610, Formula 25.12.2]{nist}
\begin{align*}
\mathrm{Li}_2(z)=-\int_0^1\frac{\log(1-tz)}{t}dt,\quad (z\in\mathbb{C}\backslash(1,\infty)),
\end{align*}
we get
\begin{align}\label{1v2}
\sum_{\beta^n=1}\mathrm{Li}_2\left(\beta v^{1/n}\right)&=-\int_0^1\frac{1}{t}\sum_{\beta^n=1}\log(1-t\beta v^{1/n})dt\nonumber\\
&=-\int_0^1\frac{\log(1-vt^n)}{t}dt=\sum_{m=1}^\infty\frac{v^m}{m}\int_0^1 t^{mn-1}dt\nonumber\\
&=\frac{1}{n}\mathrm{Li}_2(v),
\end{align}
where in the second step we used \eqref{distribution}. Equations \eqref{1v1} and \eqref{1v2} together give \eqref{novikov funct at u=1} after letting $\beta=e^{\frac{2\pi i j}{n}},\ 1\leq j\leq n$.
\end{proof}


\begin{proof}[Corollary \textup{\ref{combinations of novikov}}][]
The result follows easily upon invoking Corollary \ref{novikov funct at u=1} and \eqref{polylog x and x-1}, \eqref{euler} and the relation
\begin{align*}
\mathrm{Li}_2(z)-\mathrm{Li}_2\left(\frac{1}{1-z}\right)=\frac{1}{2}\log(1-z)\log\left(\frac{1-z}{z^2}\right)-\frac{\pi^2}{6}, \quad (z\in\mathbb{C}\backslash [1,\infty)).
\end{align*}
\end{proof}

\subsection{Proof of results related to the function $J(x)$}

\begin{proof}[Theorem \textup{\ref{fe ja}}][]
The proof is straightforward by letting $u=v=-1$ in \eqref{twoterm fe} and using \eqref{j as special case}.
\end{proof}
\begin{remark}
We here note that the relation \eqref{fe ja eqn} can also be proved by using \eqref{jxfx} and the two-term functional equation \eqref{fe2}. To the best of our knowledge, surprisingly, \eqref{fe ja eqn} does not appear in the literature.
\end{remark}

We now present a proof of the three-term functional equation of $J(x)$ by first obtaining a new representation for it.
\begin{proof}[Theorem \textup{\ref{3term thm}}][]
We invoke three-term functional equation \eqref{three-term fe} but in the form given in \eqref{constants}. We let $u=v=-1$ so that
\begin{align}\label{spe}
&\mathscr{F}(x;-1,-1)-\mathscr{F}(x-1;1,-1)+\mathscr{F}\left(\frac{x-1}{x};1,-1\right)\nonumber\\
&=\mathscr{F}(2;-1,-1)-\mathscr{F}(1;1,-1)+\mathscr{F}\left(\frac{1}{2};1,-1\right)+\left(\frac{1}{x}-\frac{1}{2}\right)\mathrm{Li}_2(-1).
\end{align}
Employing \eqref{hn to her} twice along with \eqref{j as special case} in \eqref{spe}, we deduce that, for $\mathrm{Re}(x)>1$,
\begin{align}\label{in terms}
&-J(x)-F\left(\frac{x-1}{2}\right)+F\left(x-1\right)+F\left(\frac{x-1}{2x}\right)-F\left(\frac{x-1}{x}\right)\nonumber\\
&=\mathscr{F}(2;-1,-1)-\mathscr{F}(1;1,-1)+\mathscr{F}\left(\frac{1}{2};1,-1\right)+\left(\frac{1}{x}-\frac{1}{2}\right)\mathrm{Li}_2(-1).
\end{align}
Employing Corollary \ref{novikov funct eval}, one can evaluate
$
\mathscr{F}(1;1,-1)=\frac{\pi^2}{12}-\frac{1}{2}\log^2(2),\
\mathscr{F}(2;-1,-1)=\frac{1}{48}(\pi^2-36\log^2(2)),\
\mathscr{F}\left(\frac{1}{2};1,-1\right)=\frac{5\pi^2}{48}-\frac{\log^2(2)}{4}$. Therefore, using these values along with the fact $\mathrm{Li}_2(-1)=-\pi^2/12$ in \eqref{in terms}, it turns out that
\begin{align}
J(x)+F\left(\frac{x-1}{2}\right)-F\left(x-1\right)-F\left(\frac{x-1}{2x}\right)+F\left(\frac{x-1}{x}\right)=\frac{1}{2}\log^2(2)+\frac{\pi^2}{12}\frac{x+1}{x}.\nonumber
\end{align}
Equation \eqref{jx repres} follows after invoking three-term functional equation \eqref{fe1} in the above equation.

We next prove \eqref{3tem fe for j(x)1 eqn}. 

Using \eqref{jxfx} and \eqref{jx repres}, we have
\begin{align}
J(x-1)&=F(2(x-1))-2F(x-1)+F\left(\frac{x-1}{2}\right)+\frac{\pi^2}{12(x-1)}\label{jx-1}\\
J\left(\frac{x}{x-1}\right)&=F\left(\frac{x}{x-1}\right)-F\left(\frac{1}{2(x-1)}\right)+F\left(\frac{1}{2x}\right)+\frac{\pi^2(3x-1)}{12x}+\frac{\gamma^2}{2}+\gamma_1\nonumber\\
&\quad+\frac{\log^2(2)}{2}+\mathrm{Li}_2\left(\frac{x-1}{x}\right).\label{jx-1/x}
\end{align}
Equations \eqref{jxfx}, \eqref{jx-1} and \eqref{jx-1/x} together result in
\begin{align}
&J(x)-J(x-1)+J\left(\frac{x}{x-1}\right)\nonumber\\
&=\left\{F(2x)+F\left(\frac{1}{2x}\right)\right\}-\left\{F(2(x-1))+F\left(\frac{1}{2(x-1)}\right)\right\}+F\left(\frac{x-1}{x}\right)\nonumber\\
&\quad+\left\{F\left(\frac{x-1}{x}\right)+F\left(\frac{x}{x-1}\right)\right\}+2\left\{F(x-1)-F(x)-F\left(\frac{x-1}{x}\right)\right\}\nonumber\\
&\quad+F\left(\frac{x}{2}\right)-F\left(\frac{x-1}{2}\right)+\frac{\pi^2}{4}-\frac{\pi^2}{12(x-1)}+\frac{1}{2}\gamma^2+\gamma_1+\frac{1}{2}\log^2(2)+\mathrm{Li}_2\left(\frac{x-1}{x}\right).\nonumber
\end{align}
The proof of \eqref{3tem fe for j(x)1 eqn} is now complete upon employing two and three term functional equations \eqref{fe2} and \eqref{fe1} in the above equation.

To prove \eqref{3tem fe for j(x)2 eqn}, we rewrite  \eqref{3tem fe for j(x)1 eqn} as
\begin{align}
&J(x)-J(x-1)+J\left(\frac{x}{x-1}\right)\nonumber\\
&=\left\{F\left(\frac{x}{2}\right)-F(x)\right\}-\left\{F\left(\frac{x-1}{2}\right)-F(x-1)\right\}-\left\{F(x-1)-F(x)-F\left(\frac{x-1}{x}\right)\right\}\nonumber\\
&\quad+2\mathrm{Li}_2\left(\frac{1}{x}\right)+\mathrm{Li}_2\left(\frac{x-1}{x}\right)-\frac{\pi^2(x^2+1)}{12x(x-1)}+\frac{1}{2}\gamma^2+\gamma_1\nonumber\\
&\quad+\frac{1}{2}\left(\log^2(2)+\log^2(2x)+\log^2\left(\frac{x}{x-1}\right)-\log^2(2(x-1))\right),\nonumber
\end{align}
and then invoke  \eqref{fe1} and \eqref{hn to her}.
\end{proof}

\begin{proof}[Corollary \textup{\ref{J evaluation}}][]
Letting  $u=v=-1$ in Corollary \ref{novikov funct eval at n} and using \eqref{j as special case}, we obtain
\begin{align}
J(n)=-n\ \mathrm{Li}_2\left(\frac{1}{2}\right)-\frac{1}{n}\mathrm{Li}_2(-1)+\sum_{j=1}^n\mathrm{Li}_2\left(\frac{1}{2}\left(1+ (-1)^{1/n}e^{\frac{2\pi ij}{n}}\right)\right).\nonumber
\end{align}
Now using the facts that $\mathrm{Li}_2(-1)=-\frac{1}{12}$ and $\mathrm{Li}_2(1/2)=\frac{\pi^2}{12}-\frac{1}{2}\log^2(2)$ in the above equation, we complete the proof of \eqref{j(n) eval}.

Invoking Theorem \ref{fe ja} in \eqref{j(n) eval} immediately gives \eqref{j(1/n) eval}.
\end{proof}

Let us consider the following lemma. This will help us to prove Corollary \ref{J evaluation at even}.
\begin{lemma}\label{finite sum evaluations}
Let $m\in\mathbb{N}$. We have
\begin{align}
&\mathrm{(1)}\qquad \sum_{j=0}^{m-1}\log\left(\sin\left(\frac{\pi(2j+1)}{2m}\right)\right)=(1-m)\log(2).\label{log sin eqn}\\
&\mathrm{(2)}\qquad \sum_{j=0}^{m-1}\log\left(\sin\left(\frac{\pi(2j+1)}{4m}\right)\right)=\sum_{j=0}^{m-1}\log\left(\cos\left(\frac{\pi(2j+1)}{4m}\right)\right)=\left(\frac{1}{2}-m\right)\log(2).\label{log sin cos eqn}\\
&\mathrm{(3)}\qquad \sum_{j=0}^{m-1}j\log\left(\sin\left(\frac{\pi(2j+1)}{2m}\right)\right)=-\frac{1}{2}(m-1)^2\log(2).\label{j log sin eqn}
\end{align}
\end{lemma}
\begin{proof}
Note that
\begin{align}\label{beforegrn}
\sum_{j=0}^{m-1}\log\left(\sin\left(\frac{\pi(2j+1)}{2m}\right)\right)&=\log\left(\prod_{j=0}^{m-1}\sin\left(\frac{\pi j}{m}+\frac{\pi}{2m}\right)\right).
\end{align}
From \cite[p.~41, Formula 1.392.1]{grn}, we have
\begin{align}\label{grn}
\sin(nx)=2^{n-1}\prod_{j=0}^{n-1}\sin\left(\frac{j\pi}{n}+x\right).
\end{align}
Invoking the above result in \eqref{beforegrn}, we deduce that
\begin{align}
\sum_{j=0}^{m-1}\log\left(\sin\left(\frac{\pi(2j+1)}{2m}\right)\right)&=\log\left(2^{1-m}\sin\left(\frac{\pi}{2}\right)\right)\nonumber\\
&=(1-m)\log(2).\nonumber
\end{align}
This proves \eqref{log sin eqn}.

Next,
\begin{align}\label{sin}
2\sum_{j=0}^{m-1}\log\left(\sin\left(\frac{\pi(2j+1)}{4m}\right)\right)&=\sum_{j=0}^{m-1}\log\left(\sin\left(\frac{\pi(2j+1)}{4m}\right)\right)+\sum_{j=0}^{m-1}\log\left(\sin\left(\frac{\pi(2j+1)}{4m}\right)\right)\nonumber\\
&=\sum_{j=0}^{m-1}\log\left(\sin\left(\frac{\pi(2j+1)}{4m}\right)\right)+\sum_{j=m}^{2m-1}\log\left(\sin\left(\frac{\pi(2j+1)}{4m}\right)\right)\nonumber\\
&=\sum_{j=0}^{2m-1}\log\left(\sin\left(\frac{\pi(2j+1)}{4m}\right)\right),
\end{align}
where in the second step we made the change of variable $j\to 2m-1-j$ in the second sum on the right-hand side. We now employ \eqref{grn} in \eqref{sin} so as to obtain
\begin{align}\label{sin sum}
2\sum_{j=0}^{m-1}\log\left(\sin\left(\frac{\pi(2j+1)}{4m}\right)\right)&=\log\left(2^{1-2m}\sin\left(\frac{\pi}{2}\right)\right)\nonumber\\
&=(1-2m)\log(2).
\end{align}
Upon employing the change of variable $j\to m-1-j$ below and using \eqref{sin sum}, one gets
\begin{align}
\sum_{j=0}^{m-1}\log\left(\cos\left(\frac{\pi(2j+1)}{4m}\right)\right)&=\sum_{j=0}^{m-1}\log\left(\sin\left(\frac{\pi(2j+1)}{4m}\right)\right)=\left(\frac{1}{2}-m\right)\log(2).\nonumber
\end{align}
This proves \eqref{log sin cos eqn}.

Now, again the change of variable $j\to m-1-j$ gives
\begin{align}
&2\sum_{j=0}^{m-1}j\log\left(\sin\left(\frac{\pi(2j+1)}{2m}\right)\right)\nonumber\\
&=\sum_{j=0}^{m-1}j\log\left(\sin\left(\frac{\pi(2j+1)}{2m}\right)\right)+\sum_{j=0}^{m-1}(m-1-j)\log\left(\sin\left(\frac{\pi(2j+1)}{2m}\right)\right)\nonumber\\
&=(m-1)\sum_{j=0}^{m-1}\log\left(\sin\left(\frac{\pi(2j+1)}{2m}\right)\right)\nonumber\\
&=-(m-1)^2\log(2),\nonumber
\end{align}
which follows after invoking \eqref{log sin eqn}. This completes the proof of the lemma.
\end{proof}

\begin{proof}[Corollary \textup{\ref{J evaluation at even}}][]
Letting $n=2m$ in Corollary \ref{J evaluation}, one has
\begin{align}\label{specialization}
J(2m)=\frac{\pi^2}{24m}-\frac{\pi^2m}{6}+m\log^2(2)+\sum_{j=1}^{2m}\mathrm{Li}_2\left(\frac{1}{2}\left(1+e^{\frac{\pi(2j+1)}{2m}}\right)\right).
\end{align}
Our main task is to reduce the finite sum involving dilogarithm. This is  what we do next. First observe that
\begin{align}
\sum_{j=1}^{2m}\mathrm{Li}_2\left(\frac{1}{2}\left(1+e^{\frac{\pi(2j+1)}{2m}}\right)\right)&=\sum_{j=1}^{m}\mathrm{Li}_2\left(\frac{1}{2}\left(1+e^{\frac{\pi i(2j+1)}{2m}}\right)\right)+\sum_{j=m+1}^{2m}\mathrm{Li}_2\left(\frac{1}{2}\left(1+e^{\frac{\pi i(2j+1)}{2m}}\right)\right)\nonumber\\
&=\sum_{j=1}^{m}\mathrm{Li}_2\left(\frac{1}{2}\left(1+e^{\frac{\pi i(2j+1)}{2m}}\right)\right)+\sum_{j=1}^{m}\mathrm{Li}_2\left(\frac{1}{2}\left(1-e^{\frac{\pi  i(2j+1)}{2m}}\right)\right),\nonumber
\end{align}
where we made the change of variable $j\to j+m$ in the last sum. First add and subtract corresponding $j=0$ term to both sums and then observe that the term that we subtracted can be cancelled by the $j=m$ terms of the summations. Therefore, we get
\begin{align}\label{addsub}
\sum_{j=1}^{2m}\mathrm{Li}_2\left(\frac{1}{2}\left(1+e^{\frac{\pi(2j+1)}{2m}}\right)\right)&=\sum_{j=0}^{m-1}\mathrm{Li}_2\left(\frac{1}{2}\left(1+e^{\frac{\pi i(2j+1)}{2m}}\right)\right)+\sum_{j=0}^{m-1}\mathrm{Li}_2\left(\frac{1}{2}\left(1-e^{\frac{\pi i(2j+1)}{2m}}\right)\right).
\end{align}
Using \eqref{euler} once with letting  $z=\frac{1+a}{2}$  and once with  $z=\frac{1-a}{2}$ and then adding both resulting expressions, we obtain
\begin{align*}
\mathrm{Li}_2\left(\frac{1+a}{2}\right)+\mathrm{Li}_2\left(\frac{1-a}{2}\right)=-\log\left(\frac{1+a}{2}\right)\log\left(\frac{1-a}{2}\right)+\frac{\pi^2}{6}.
\end{align*}
Letting $a=e^{\frac{\pi i(2j+1)}{2m}}$ in the above expression, we see that
\begin{align}\label{li2+li2}
&\mathrm{Li}_2\left(\frac{1}{2}\left(1+e^{\frac{\pi i(2j+1)}{2m}}\right)\right)+\mathrm{Li}_2\left(\frac{1}{2}\left(1-e^{\frac{\pi i(2j+1)}{2m}}\right)\right)\nonumber\\
&=\frac{\pi^2}{6}-\log\left(-i e^{\frac{\pi i(2j+1)}{4m}}\sin\left(\frac{\pi(2j+1)}{4m}\right)\right)\log\left(e^{\frac{\pi i(2j+1)}{4m}}\cos\left(\frac{\pi(2j+1)}{4m}\right)\right)\nonumber\\
&=\frac{\pi^2}{6}-\left\{\log\left(e^{\frac{\pi i(2j+1)}{4m}-\frac{\pi i}{2}}\right)+\log\left(\sin\left(\frac{\pi(2j+1)}{4m}\right)\right)\right\}\nonumber\\
&\qquad\times\left\{\log\left(e^{\frac{\pi i(2j+1)}{4m}}\right)+\log\left(\cos\left(\frac{\pi(2j+1)}{4m}\right)\right)\right\}\nonumber\\
&=\frac{\pi^2}{6}-\left(\frac{\pi i(2j+1)}{4m}-\frac{\pi i}{2}\right) \left(\frac{\pi i(2j+1)}{4m}\right)-\left(\frac{\pi i(2j+1)}{4m}-\frac{\pi i}{2}\right)\log\left(\cos\left(\frac{\pi(2j+1)}{4m}\right)\right)\nonumber\\
&\quad-\left(\frac{\pi i(2j+1)}{4m}\right)\log\left(\sin\left(\frac{\pi(2j+1)}{4m}\right)\right)-\log\left(\sin\left(\frac{\pi(2j+1)}{4m}\right)\right)\log\left(\cos\left(\frac{\pi(2j+1)}{4m}\right)\right).
\end{align}
Equations \eqref{addsub} and \eqref{li2+li2} together yield
\begin{align}
&\sum_{j=1}^{2m}\mathrm{Li}_2\left(\frac{1}{2}\left(1+e^{\frac{\pi(2j+1)}{2m}}\right)\right)\nonumber\\
&=\frac{\pi^2m}{6}-\frac{\pi^2}{48m}(2m^2+1)-\frac{\pi i}{4m}\sum_{j=0}^{m-1}\left\{\log\left(\sin\left(\frac{\pi(2j+1)}{4m}\right)\right)+\log\left(\cos\left(\frac{\pi(2j+1)}{4m}\right)\right)\right\}\nonumber\\
&\quad+\frac{\pi i}{2}\sum_{j=0}^{m-1} \log\left(\cos\left(\frac{\pi(2j+1)}{4m}\right)\right)-\sum_{j=0}^{m-1}\log\left(\sin\left(\frac{\pi(2j+1)}{4m}\right)\right) \log\left(\cos\left(\frac{\pi(2j+1)}{4m}\right)\right)\nonumber\\
&\quad-\frac{\pi i}{2m}\sum_{j=0}^{m-1}j\left\{ \log\left(\sin\left(\frac{\pi(2j+1)}{4m}\right)\right)+\log\left(\cos\left(\frac{\pi(2j+1)}{4m}\right)\right)\right\}\nonumber\\
&=\frac{\pi^2m}{8}-\frac{\pi^2}{48m}-\frac{\pi i}{4m}\sum_{j=0}^{m-1}\log\left(\frac{1}{2}\sin\left(\frac{\pi(2j+1)}{2m}\right)\right)+\frac{\pi i}{2}\sum_{j=0}^{m-1} \log\left(\cos\left(\frac{\pi(2j+1)}{4m}\right)\right)\nonumber\\
&\quad-\frac{\pi i}{2m}\sum_{j=0}^{m-1}j\log\left(\frac{1}{2}\sin\left(\frac{\pi(2j+1)}{2m}\right)\right)-\sum_{j=0}^{m-1}\log\left(\sin\left(\frac{\pi(2j+1)}{4m}\right)\right) \log\left(\cos\left(\frac{\pi(2j+1)}{4m}\right)\right).\nonumber
\end{align}
in the last step we used the fact $\sin(A)\cos(A)=\frac{1}{2}\sin(2A)$. The above expression can be further simplified to
\begin{align}\label{therek}
&\sum_{j=1}^{2m}\mathrm{Li}_2\left(\frac{1}{2}\left(1+e^{\frac{\pi(2j+1)}{2m}}\right)\right)\nonumber\\
&=\frac{\pi^2m}{8}-\frac{\pi^2}{48m}+\frac{\pi im}{4}\log(2)-\frac{\pi i}{4m}\sum_{j=0}^{m-1}\log\left(\sin\left(\frac{\pi(2j+1)}{2m}\right)\right)\nonumber\\
&\quad+\frac{\pi i}{2}\sum_{j=0}^{m-1} \log\left(\cos\left(\frac{\pi(2j+1)}{4m}\right)\right)-\frac{\pi i}{2m}\sum_{j=0}^{m-1}j\log\left(\sin\left(\frac{\pi(2j+1)}{2m}\right)\right)\nonumber\\
&\quad-\sum_{j=0}^{m-1}\log\left(\sin\left(\frac{\pi(2j+1)}{4m}\right)\right) \log\left(\cos\left(\frac{\pi(2j+1)}{4m}\right)\right).
\end{align}
Invoking Lemma \ref{finite sum evaluations} in \eqref{therek}, we are led to 
\begin{align}\label{li2 to log}
&\sum_{j=1}^{2m}\mathrm{Li}_2\left(\frac{1}{2}\left(1+e^{\frac{\pi(2j+1)}{2m}}\right)\right)\nonumber\\
&=\frac{\pi^2m}{8}-\frac{\pi^2}{48m}-\sum_{j=0}^{m-1}\log\left(\sin\left(\frac{\pi(2j+1)}{4m}\right)\right)\log\left(\cos\left(\frac{\pi(2j+1)}{4m}\right)\right).
\end{align}
Substituting \eqref{li2 to log} in \eqref{specialization}, we arrive at \eqref{J evaluation at even eqn}.

Other part is simply a consequence of \eqref{J evaluation at even eqn} and \eqref{fe ja}.
\end{proof}

We note that Anthony and Batir \cite[p.~12]{ab} obtained the following equivalent version of \eqref{J evaluation at even eqn}:
\begin{align}\label{ab result}
J(2m)&=\frac{\pi^2}{48m}-\frac{\pi^2m}{24}+\log^2(2)+\frac{1}{2}\sum_{j=0}^{2m-1}\log^2\left(2\sin\left(\frac{\pi(2j+1)}{4m}\right)\right)\nonumber\\
&\qquad-\frac{1}{2}\sum_{j=0}^{m-1}\log^2\left(2\sin\left(\frac{\pi(2j+1)}{2m}\right)\right).
\end{align}
Equation \eqref{ab result} can be reduced to \eqref{J evaluation at even eqn}. It is shown next. Observe that
\begin{align}\label{ab1}
&\sum_{j=0}^{2m-1}\log^2\left(2\sin\left(\frac{\pi(2j+1)}{4m}\right)\right)\nonumber\\
&=\sum_{j=0}^{m-1}\log^2\left(2\sin\left(\frac{\pi(2j+1)}{4m}\right)\right)+\sum_{j=0}^{m-1}\log^2\left(2\cos\left(\frac{\pi(2j+1)}{4m}\right)\right).
\end{align}
Using the duplication formula $\sin\frac{\pi(2j+1)}{2m}=2\sin\frac{\pi(2j+1)}{4m}\cos\frac{\pi(2j+1)}{4m}$, we have 
\begin{align}
&\log^2\left(2\sin\left(\frac{\pi(2j+1)}{4m}\right)\right)+\log^2\left(2\cos\left(\frac{\pi(2j+1)}{4m}\right)\right)-\log^2\left(2\sin\left(\frac{\pi(2j+1)}{2m}\right)\right)\nonumber\\
&=-2\log^2(2)-2\log(2)\log^2\left(\sin\left(\frac{\pi(2j+1)}{4m}\right)\right)-2\log(2)\log^2\left(\cos\left(\frac{\pi(2j+1)}{4m}\right)\right)\nonumber\\
&\qquad-2\log\left(\sin\left(\frac{\pi(2j+1)}{4m}\right)\right)\log\left(\cos\left(\frac{\pi(2j+1)}{4m}\right)\right).\nonumber
\end{align}
Now taking sum $\sum_{j=0}^{m-1}$ on both sides of the above equation and using \eqref{log sin cos eqn} and \eqref{ab1}, we see that
\begin{align*}
&\frac{1}{2}\sum_{j=0}^{2m-1}\log^2\left(2\sin\left(\frac{\pi(2j+1)}{4m}\right)\right)-\frac{1}{2}\sum_{j=0}^{m-1}\log^2\left(2\sin\left(\frac{\pi(2j+1)}{2m}\right)\right)\nonumber\\
&=(m-1)\log^2(2)-\sum_{j=0}^{m-1}\log\left(\sin\left(\frac{\pi(2j+1)}{4m}\right)\right)\log\left(\cos\left(\frac{\pi(2j+1)}{4m}\right)\right).
\end{align*}
After substituting the above evaluation in \eqref{ab result}, it turns out that \eqref{ab result} reduces to \eqref{J evaluation at even eqn}.


\subsection{Proofs of results on the function $T(x)$}

We first present the proof of the fact that $T(x)$ can be written as linear combination of the Herglotz-Zagier-Novikov function $\mathscr{F}\left(x;u,v\right)$.
\begin{proof}[Theorem \textup{\ref{representation of T(x)}}][]
Note that, by using the definition of $\mathscr{F}\left(x;u,v\right)$ from \eqref{novikov function} and simplifying, we see that
\begin{align*}
&\mathscr{F}\left(x;i,i\right)+\mathscr{F}\left(x;-i,-i\right)-\mathscr{F}\left(x;i,-i\right)-\mathscr{F}\left(x;-i,i\right)\\
&=2i\int_0^1\frac{1}{t^2+1}\left(\log(1-it^x)-\log(1+it^x)\right)dt\\
&=-2i\int_0^1\frac{1}{t^2+1}\sum_{n=1}^\infty\frac{i^nt^{nx}}{n}\left(1-(-1)^n\right)dt\\
&=-4i\int_0^1\frac{1}{t^2+1}\sum_{n=0}^\infty\frac{i^{2n+1}t^{(2n+1)x}}{2n+1}\\
&=4\int_0^1\frac{\mathrm{tan}^{-1}(t^x)}{t^2+1},
\end{align*}
which follows upon using the fact that $\mathrm{tan}^{-1}(t)=\sum_{n=0}^\infty\frac{(-1)^{n}}{2n+1}t^{2n+1}$. We arrive at \eqref{representation of T(x) eqn} after employing \eqref{T(x) defn}.
\end{proof}

We next prove the relation between the functions $J(x),\ T(x)$ and $\mathscr{F}(x;u,v)$.
\begin{proof}[Theorem \textup{\ref{relation b/w j t f}}][]
We invoke \eqref{sum formula for Jx eqn2} so that
\begin{align}\label{fii}
\mathscr{F}(x;i,i)=\mathscr{F}\left(\frac{x}{2};i,-1\right)-\mathscr{F}(x;i,-i)=-J(x)-\mathscr{F}\left(\frac{x}{2};-i,-1\right)-\mathscr{F}(x;i,-i),
\end{align}
the last equality follows by first using \eqref{sum formula for Jx eqn1} and then \eqref{j as special case}. Similarly,
\begin{align}\label{f-i-i}
\mathscr{F}(x;-i,-i)=-J(x)-\mathscr{F}\left(\frac{x}{2};i,-1\right)-\mathscr{F}(x;-i,i).
\end{align}
Observe that $-\mathscr{F}\left(\frac{x}{2};i,-1\right)-\mathscr{F}\left(\frac{x}{2};-i,-1\right)=J(x)$. Employing this fact along with \eqref{fii} and \eqref{f-i-i} in \eqref{representation of T(x) eqn}, we complete the proof of \eqref{relation b/w j t f eqn}.
\end{proof}

As an applications of Theorem \ref{twoterm fe} and \ref{representation of T(x)}, we  now get the functional equation of $T(x)$.
\begin{proof}[Theorem \textup{\ref{fe for T(x)}}][]
Invoking Theorem \ref{representation of T(x)} twice, we obtain
\begin{align}
T(x)+T\left(\frac{1}{x}\right)&=\frac{1}{4}\left\{\left[\mathscr{F}\left(x;i,i\right)+\mathscr{F}\left(\frac{1}{x};i,i\right)\right]+\left[\mathscr{F}\left(x;-i,-i\right)+\mathscr{F}\left(\frac{1}{x};-i,-i\right)\right]\right.\nonumber\\
&\left.-\left[\mathscr{F}\left(x;i,-i\right)+\mathscr{F}\left(\frac{1}{x};-i,i\right)\right]-\left[\mathscr{F}\left(x;-i,i\right)+\mathscr{F}\left(\frac{1}{x};i,-i\right)\right]\right\}.\nonumber
\end{align}
Employing the functional equation \eqref{twoterm fe} for each square bracket in the above equation, we deduce that
\begin{align*}
&T(x)+T\left(\frac{1}{x}\right)\\
&=-\frac{1}{4}\left\{\log^2(1-i)+\log^2(1+i)-2\log(1-i)\log(1+i)\right\}\nonumber\\
&=-\frac{1}{4}\left\{\log(1-i)[\log(1-i)-\log(1+i)]-\log(1+i)[\log(1-i)-\log(1+i)]\right\}\nonumber\\
&=-\frac{1}{4}\log^2\left(\frac{1-i}{1+i}\right)=\frac{\pi^2}{16}.
\end{align*}
This completes the proof of \eqref{fe for T(x) eqn}.
\end{proof}

\begin{remark}
The functional equation \eqref{fe for T(x) eqn} for $T(x)$ can also be proved directly by adopting the same arguments as we have given for the proof of \eqref{twoterm fe}.
\end{remark}

\begin{proof}[Corollary \textup{\ref{T(n) eval thm}}][]
Equation \eqref{T(n) eval} follows readily from Corollary \ref{novikov funct eval at n} and Theorem \ref{representation of T(x)} and upon simplifying. 

Theorem \ref{fe for T(x)} and \eqref{T(n) eval} together yield \eqref{T(1/n) eval}.
\end{proof}

\section{Tables containing special values of $J(x)$ and $T(x)$}\label{tables}
As mentioned earlier, Herglotz \cite{her1} and Radchenko and Zagier \cite{raza} have found evaluations for the functions $J(x)$, while Muzaffar and Williams \cite{muzwil} have provided evaluations for both functions $J(x)$ and $T(x)$. In the tables below, we provide some of the new special values of these functions employing our Theorem \ref{fe ja} and  Theorem \ref{fe for T(x)} along with \cite[Theorems 7 \& 10 ]{muzwil}. More such evaluations can be found using our Theorem \ref{fe ja} and  Theorem \ref{fe for T(x)} and the results of Herglotz \cite{her1} and Radchenko and Zagier \cite{raza}.

\begin{center}
\textbf{Table 1:}\vspace{2mm}
\begin{tabular}{ |p{2cm}||p{12.5cm}|}
 \hline
 \multicolumn{2}{|c|}{\textbf{Special values of $J(x)$}} \\
 \hline
 \hspace{0.6cm} $x$ & \hspace{5cm} $J(x)$ \vspace{1mm}\\
 \hline
 $2-\sqrt{3}$ & $\log^2(2)-\frac{\pi^2}{12}(1-\sqrt{3})-\log(2)\log(1+\sqrt{3})$\\
\hline
$3-\sqrt{8}$   & $\frac{1}{2}\log^2(2)-\frac{\pi^2}{24}(3-\sqrt{32})-\frac{3}{4}\log(2)\log(3+\sqrt{8})$    \\
\hline
$4-\sqrt{15}$ & $\log^2(2)-\frac{\pi^2}{12}(2-\sqrt{15})-\log\left(\frac{1+\sqrt{5}}{2}\right)\log(2+\sqrt{3})-\log(2)\log(\sqrt{3}+\sqrt{5})$\\
\hline
$5-\sqrt{24}$&$\frac{1}{2}\log^2(2)-\frac{\pi^2}{24}(5-\sqrt{96})-\frac{1}{2}\log(1+\sqrt{2})\log(2+\sqrt{3})-\frac{3}{4}\log(2)\log(5+\sqrt{24})$ \\
\hline
$6-\sqrt{35}$ & $\log^2(2)-\frac{\pi^2}{12}(3-\sqrt{35})-\log\left(\frac{1+\sqrt{5}}{2}\right)\log(8+3\sqrt{7})-\log(2)\log(\sqrt{5}+\sqrt{7})$\\
\hline
$8-\sqrt{63}$ & $\log^2(2)-\frac{\pi^2}{12}(4-\sqrt{63})-\log\left(\frac{5+\sqrt{21}}{2}\right)\log(2+\sqrt{3})-\log(2)\log(3+\sqrt{7})$\\
\hline
$11-\sqrt{120}$ & $\frac{1}{2}\log^2(2)-\frac{\pi^2}{24}(11-\sqrt{480})-\frac{1}{2}\log(1+\sqrt{2})\log(4+\sqrt{15})-\frac{1}{2}\log(2+\sqrt{3})\log(3+\sqrt{10})-\frac{1}{2}\log\left(\frac{1+\sqrt{5}}{2}\right)\log(5+\sqrt{24})-\frac{3}{4}\log(2)\log(11+\sqrt{120})$ \\
\hline
$12-\sqrt{143}$ & $\log^2(2)-\frac{\pi^2}{12}(6-\sqrt{143})-\log\left(\frac{3+\sqrt{13}}{2}\right)\log(10+3\sqrt{11})-\log(2)\log(\sqrt{11}+\sqrt{13})$\\
\hline
$13-\sqrt{168}$&$\frac{1}{2}\log^2(2)-\frac{\pi^2}{24}(13-\sqrt{672})-\frac{1}{2}\log(1+\sqrt{2})\log\left(\frac{5+\sqrt{21}}{2}\right)-\frac{1}{4}\log(2+\sqrt{3})\log(15+\sqrt{224})-\frac{1}{4}\log(5+\sqrt{24})\log(8+\sqrt{63})-\frac{3}{4}\log(2)\log(13+\sqrt{168})$\\
\hline
$14-\sqrt{195}$ & $\log^2(2)-\frac{\pi^2}{12}(7-\sqrt{195})-\log\left(\frac{1+\sqrt{5}}{2}\right)\log(25+4\sqrt{39})-\log\left(\frac{3+\sqrt{13}}{2}\right)\log(4+\sqrt{15})-\log(2)\log(\sqrt{15}+\sqrt{13})$ \\
\hline
\end{tabular}
\end{center}
\begin{center}
\textbf{Table 2:}\vspace{2mm}
\begin{tabular}{ |p{2cm}||p{12.5cm}|}
 \hline
 \multicolumn{2}{|c|}{\textbf{Special values of $T(x)$}} \\
 \hline
 \hspace{0.8cm}$x$ & \hspace{5cm} $T(x)$ \vspace{1mm}\\
 \hline
 $3-\sqrt{8}$   & $\frac{1}{16}\big\{\pi^2-\log(2)\log(3+\sqrt{8})\big\}$    \\
\hline
$5-\sqrt{24}$   & $\frac{1}{16}\big\{\pi^2+\log(2)\log(5+\sqrt{24})\big\}-\frac{1}{8}\log(1+\sqrt{2})\log(2+\sqrt{3})$    \\
\hline
$11-\sqrt{120}$ & $\frac{1}{16}\big\{\pi^2-\log(2)\log(11+\sqrt{120})\big\}+\frac{1}{8}\log(1+\sqrt{2})\log(4+\sqrt{15})+\frac{1}{8}\log(2+\sqrt{3})\log(3+\sqrt{10})-\frac{3}{8}\log\left(\frac{1+\sqrt{5}}{2}\right)\log(5+\sqrt{24})$ \\
\hline
$13-\sqrt{168}$ & $\frac{1}{16}\big\{\pi^2+\log(2)\log(13+\sqrt{168})+6\log(1+\sqrt{2})\log\left(\frac{5+\sqrt{21}}{2}\right)-\log(2+\sqrt{3})\log(15+\sqrt{224})-\log(5+\sqrt{24})\log(8+\sqrt{63})\big\}$\\
\hline
\end{tabular}
\end{center}

\medskip

\noindent {\bf{Acknowledgement  }}\, We would like to thank the referees for numerous helpful comments and suggestions which greatly improved the exposition of this paper. The first author was partially supported by   NRF$2022R1A2B5B0100187111 $, NRF$2021R1A6A1A1004294412$ and the second author was partially supported by the grant  IBS-R$003$-D1  of the IBS-CGP, POSTECH, South Korea and the Fulbright-Nehru Postdoctoral Fellowship Grant 2846/FNPDR/2022. This work was intiated when the second author was a postdoc fellow at the IBS center for Geometry and Physics, POSTECH, South Korea.

\end{document}